\documentclass[11pt,reqno]{amsart}

\usepackage{amsthm}
\usepackage{amssymb}
\usepackage{mathrsfs}
\usepackage{amsmath}
\usepackage{latexsym}
\usepackage[CJKbookmarks=true]{hyperref}
\usepackage{graphicx,tikz}
\usepackage[all]{xy}
\usepackage{ytableau}
\usepackage{graphicx}
\usepackage{enumerate}
\usepackage{eucal}

\numberwithin{equation}{section}
\newtheorem{Th}{Theorem}[section]
\newtheorem{Lemma}[Th]{Lemma}
\newtheorem{Coro}[Th]{Corollary}
\newtheorem{Prop}[Th]{Proposition}
\newtheorem{Eg}[Th]{Example}

\newtheorem{Thm}{Theorem}

\theoremstyle{definition}
\newtheorem{Def}[Th]{Definition}

\theoremstyle{remark}
\newtheorem{Rmk}[Th]{Remark}

\title[Automorphism of $QH^*_{\bbT}(T^*\calB)$ and applications]{Automorphisms of the Quantum Cohomology of the Springer Resolution and Applications} 

\author{Changzheng Li}
\address[Changzheng Li]{School of Mathematics, Sun Yat-sen University, Guangzhou 510275, P.R. China}
\email{lichangzh@mail.sysu.edu.cn}

\author{Changjian Su}
\address[Changjian Su]{Yau Mathematical Sciences Center, Tsinghua University, Beijing, China}
\email{changjiansu@mail.tsinghua.edu.cn}

\author{Rui Xiong}
\address[Rui Xiong]{Department of Mathematics and Statistics, University of Ottawa, 150 Louis-Pasteur, Ottawa, ON, K1N 6N5, Canada}
\email{rxion043@uottawa.ca}

\def\ft{\mathfrak{t}}
\def\bbC{\mathbb{C}}
\def\bbF{\mathbb{F}}
\def\bbG{\mathbb{G}}
\def\bbP{\mathbb{P}}
\def\bbT{\mathbb{T}}
\def\bbZ{\mathbb{Z}}
\def\calB{\mathcal{B}}
\def\calD{\mathcal{D}}
\def\calH{\mathcal{H}}
\def\calL{\mathcal{L}}
\def\calN{\mathcal{N}}
\def\calO{\mathcal{O}}
\def\calP{\mathcal{P}}
\def\tdlim{\lim\nolimits^{\mathrm{td}}}

\DeclareMathOperator{\pt}{pt}
\DeclareMathOperator{\id}{id}
\DeclareMathOperator{\Stab}{Stab}
\DeclareMathOperator{\Eff}{Eff}

\DeclareMathOperator{\St}{St}
\DeclareMathOperator{\reg}{reg}
\DeclareMathOperator{\Sym}{Sym}
\DeclareMathOperator{\SL}{SL}
\DeclareMathOperator{\Hom}{Hom}

\DeclareMathOperator{\cl}{cl}
\DeclareMathOperator{\CM}{CM}
\DeclareMathOperator{\Dun}{Dun}
\DeclareMathOperator{\ev}{ev}

\let\oldvee\vee
\def\vee{{\mathchoice{\oldvee}{\oldvee}{\oldvee}{\mathsf{v}}}}
\def\o#1{\hspace{0.1em}\overline{\hspace{-0.1em}#1\hspace{-0.1em}}\hspace{0.1em}}

\begin{document}

\maketitle

\begin{abstract}
    In this paper, we introduce quantum Demazure--Lusztig operators acting by ring automorphisms on the equivariant quantum cohomology of the Springer resolution. Our main application is a presentation of the torus-equivariant quantum cohomology in terms of generators and relations. We provide explicit descriptions for the classical types. We also recover Kim's earlier results for the complete flag varieties by taking the Toda limit.
\end{abstract}

\keywords{Keywords: Quantum Demazure--Lusztig operators, Springer resolution, Quantum cohomology}


\section{Introduction}

The quantum cohomology ring $QH^*(X)$ of a complex projective manifold $X$ is a deformation of the classical cohomology ring $H^*(X)=H^*(X, \mathbb{C})$ by incorporating the Gromov--Witten invariants. 
When $X$ is a non-compact algebraic variety equipped with a nice reductive group $\mathbb{G}$ action, the equivariant quantum product can still be defined via localizations (see e.g. \cite{MR3184181}). An interesting example of such varieties is the symplectic resolution, which usually comes in pairs called symplectic duals \cite{J22}. 
The symplectic duality is a 3-dimensional mirror symmetry, generalizing well-known structures in geometric representation theory \cite{MR3594665}. 
Okounkov and his collaborators initiated the study of the symplectic duality via the enumerative geometry of the symplectic resolution, which is deeply related to various structures in geometry, representation theory, and mathematical physics \cite{MR3966746}.
There have been studies of $QH^*_{\bbG}(X)$ for various examples of symplectic resolutions \cite{MR2587340, MR2782198,MR3439689, MR3095147,MR3951025,Dan22, MR3421784, MR4295090}.

The Springer resolution, which is the cotangent bundle of the complete flag variety $\calB$ of a complex semisimple, simply-connected Lie group $G$, is the most classical example of the symplectic resolutions. 
It is a resolution of the nilpotent cone in $\mathfrak{g}=\operatorname{Lie}(G)$ and admits a natural action by $\mathbb{G}=G\times \mathbb{C}^*$, where 
$\bbC^*$ dilates the cotangent fibers. 
Let $T$ be a maximal torus of $G$ with Lie algebra $\mathfrak{t}$, and denote $\bbT:=T\times \mathbb{C}^*$. 
The $\bbT$-equivariant quantum cohomology ring $QH^*_{\bbT}(T^*\calB)$ can be defined as the deformation 
  $$QH^*_{\bbT}(T^*\calB)=\big(H^*_{\bbT}(T^*\calB)\otimes \calO(T^\vee_{\reg}), \,*\,\big),$$ 
due to the equivariant quantum Chevalley formula proved by Braverman, Maulik, and Okounkov \cite[Theorem 3.2]{MR2782198}. Here $T^\vee_{\reg}$ is the complement of the union of root hyperplanes in the complex dual torus $T^\vee$ (see Section \ref{sec:Pre}). 

Our first main theorem reveals a surprising symmetry in the equivariant quantum cohomology ring of $T^*\calB$.
To explain this symmetry, let us start with the classical Springer theory. Recall that the degenerate affine Hecke algebra $\calH_\hbar$ is a $\bbC[\hbar]$-algebra generated by $\Sym(\ft^*)$ and the group algebra $\bbC[W]$, subject to interacting relations between them.
As will be reviewed in Section \ref{sec: springer},
$\calH_\hbar$ acts on the $\mathbb{T}$-equivariant cohomology $H^*_{\mathbb{T}}(T^*\calB)$, through an isomorphism of convolution algebras \cite{MR2838836, MR972345}.
The operators corresponding to $w\in W\subset \calH_\hbar$ are usually referred to as the Demazure--Lusztig operators \cite{MR1649626}. 
On the other hand, there is a natural Weyl group action on the coefficient ring $\calO(T^\vee_{\reg})$ induced by the Weyl group action on $T^\vee_{\reg}$. The \emph{quantum Demazure--Lusztig operators} $T_w$ over $QH^*_{\mathbb{T}}(T^*\calB)=H^*_{\mathbb{T}}(T^*\calB)\otimes \calO(T^\vee_{\reg})$ is the diagonal tensor product of these two actions, c.f. Definition \ref{defofQDLop}. 
Our first main theorem is the following unexpected result, for which we refer to Example \ref{eg:autoP1} for a quick illustration with  $T^*\bbP^1$.

\begin{Thm}[Theorem \ref{Swisring}] Every quantum Demazure--Lusztig operator is a ring automorphism of $QH^*_{\bbT}(T^*\calB)$. 
\end{Thm}

Observe that the usual Demazure--Lusztig operator interacts with the equivariant cup product via Leibniz-type formulas \cite[\S{}12]{MR1649626}, which is never a ring automorphism except for the trivial case. 
As the quantum Demazure--Lusztig operators that we have defined create a pole at the origin of the quantum variables, they cannot be projected down to the equivariant cohomology. 
The key ingredient in our proof is the stable basis introduced by Maulik and Okounkov \cite{MR3951025}.
We achieve this aim by direct computations using an explicit formula about the multiplication of a divisor class with the stable basis due to the second named author \cite{MR3439689}. The proof uses some crucial structures appearing in the cotangent bundle situation.

As an application of the above theorem, we obtain a ring presentation for $QH^*_{\bbT}(T^*\calB)$.
For any dominant weight $\lambda$, we define in \eqref{eq:defofThetalambda} a square matrix $\Theta(\lambda)$ with values in $QH^2_{\bbT}(T^*\calB)$, inspired from the quantum Chevalley formula. 
It leads to an effective way of finding relations in the equivariant quantum cohomology, due to an observation in communicative algebra as in Lemma \ref{lem: relation}. 
Via a Deformation Principle in Propositon \ref{Mainprop}, 
we obtain our second main result.

\begin{Thm}[Theorem \ref{qHpre}]\label{ThmB}
The $\bbT$-equivariant quantum cohomology $QH_{\bbT}^*(T^*\calB)$ is generated by divisors with the relations 
$\operatorname{tr}\big(\Theta(\lambda)^k\big)=\sum_{\mu\in W\!\lambda} \mu^k$
for all dominant weights $\lambda$ and all $k\geq 1$. 
\end{Thm}

The above ring representation has an application in the study of the Calogero--Moser system, which is an integrable system arising from the one-dimensional many-body problems (see e.g. \cite{MR2296754, MR3817553} and references therein). 
Combining Theorem \ref{ThmB} and \cite[Theorem 3.2]{MR2782198}, we obtain an explicit description of the classical trigonometric Calogero--Moser map in Corollary \ref{thm:classicalcm}. 

We are going to provide an explicit ring presentation of $QH_{\bbT}^*(T^*\calB)$ in all classical types. Instead of using the relations in Theorem \ref{ThmB} directly, we compute the characteristic polynomial of $\Theta(\lambda)$.
For instance, in type $A_{n-1}$, 
the flag variety $\calB=\mathcal{F}\ell_n$ parameterizes complete flags in $\bbC^n$ and 
the Springer resolution $T^*\mathcal{F}\ell_n$ parameterizes pairs $(A,\phi)$ for $A$ an $n\times n$ nilpotent matrix and $\phi\in \mathcal{F}\ell_n$ a complete flag of $A$-invariant subspaces.
Among the family of $\Theta(\lambda)$, We can find one matrix of the form
\begin{equation}\label{intro:eq:MatA}
\left[
\begin{matrix}
\chi_1 &
\frac{\hbar}{1-q_1/q_2}&
\frac{\hbar}{1-q_1/q_3}&
\cdots &
\frac{\hbar}{1-q_1/q_n}\\[1ex]
\frac{\hbar}{1-q_2/q_1} &
\chi_2 &
\frac{\hbar}{1-q_2/q_3} &
\cdots &
\frac{\hbar}{1-q_2/q_n}\\[1ex]
\frac{\hbar}{1-q_3/q_1} &
\frac{\hbar}{1-q_3/q_2} &
\chi_3 &
\cdots &
\frac{\hbar}{1-q_3/q_n}\\[1ex]
\vdots&\vdots&\vdots&\ddots&\vdots\\[1ex]
\frac{\hbar}{1-q_n/q_1} &
\frac{\hbar}{1-q_n/q_2} &
\frac{\hbar}{1-q_n/q_3} &
\cdots &
\chi_n
\end{matrix}\right],
\end{equation}
where
$\chi_i=x_i+\hbar\sum_{a<i}
\frac{q_a/q_i}{1-q_a/q_i}
-\hbar\sum_{i<b}
\frac{q_i/q_b}{1-q_i/q_b}$. 
The matrix \eqref{intro:eq:MatA} happens to coincide with the one in \cite[\S{}2.8]{MR2296754}
The presentation for the $\bbT$-equivariant quantum cohomology of $T^*\calB$ can be stated as follows,

\begin{Thm}[Theorem \ref{Th:typeApre}] The torus-equivariant quantum cohomology of $T^*\mathcal{F}\ell_n$ is generated by divisors $x_1,\ldots,x_n$ with relations $\mathcal{E}_k(\chi)-e_k(t)$ for $k=1,\ldots,n$. Here $\mathcal{E}_k(\chi)$ are the coefficients of the characteristic polynomial of the matrix \eqref{intro:eq:MatA}.
\end{Thm}
 
We will give a combinatorial description of $\mathcal{E}_k(\chi)$ in Theorem \ref{combforEk} via matching over $k$-subsets of $\{1,\ldots,n\}$, which coincides with Feynman's elementary computation \cite{Polychronakos2019}. 
The computations for other classical types are given in Theorem \ref{prefortypeBC} (type $B$ and type $C$) and Theorem \ref{prefortypeD} (type $D$). 

There is another application of our explicit description of $QH^*_{\bbT}(T^*\calB)$. As explained in \cite[\S{}8]{MR2782198}, we can recover $QH^*_T(\calB)$ by taking the Toda limit. In the context of quantum Schubert calculus and mirror symmetry, it is an important problem to provide a precise description of the ring presentation of the ($T$-equivariant) quantum cohomology of $\calB$ or, more generally, of a partial flag variety. 
This problem has been solved for partial flag varieties of type $A$ \cite{MR1328256, MR1337131, MR1695799} and for the complete flag variety $\calB$ of general Lie type by Kim \cite{MR1680543}. However, a precise description for partial flag varieties of non-$A$ type in general is not yet available. 
By taking the Toda limit, we are able to reobtain the presentation of $QH^*_{T}(\calB)$ for classical Lie types and discuss its connection with Kim's description.
For type $A_{n-1}$, the matrix \eqref{intro:eq:MatA} (after conjugation) converges, under the Toda limit, to the well-known tridiagonal matrix in quantum Schubert calculus:
\begin{equation}\label{intro:eq:MatAcomp}
\left[
\begin{matrix}
x_1   & -1  & 0 & \cdots & 0\\[1ex]
q_1/q_2 & x_2  &-1 & \cdots & 0\\[1ex]
0    & q_2/q_3 &x_3 & \cdots & 0\\[1ex]
\vdots&\vdots&\vdots&\ddots&\vdots\\[1ex]
0& 0& 0& \cdots & x_n
\end{matrix}\right].
\end{equation}
This recovers the presentation obtained by Givental and Kim \cite{MR1328256}. 
It would be interesting to establish a direct (type-free) connection between both descriptions. 
Our approach to finding the ring presentation of $QH^*_{\bbT}(T^*\calB)$ may be generalized to the cotangent bundle of a partial flag variety. Hence, it is potentially useful in solving the aforementioned problem for the partial flag varieties.

The paper is organized as follows. 
In Section \ref{sec:Pre}, we review the background of equivariant quantum cohomology of the Springer resolution $T^*\calB$. 
In Section \ref{sec:Aut}, we define quantum Demazure--Lusztig operators and show that the operators act on the equivariant quantum cohomology of $T^*\calB$ by ring automorphisms. 
In Section \ref{sec:Pres}, we give the ring presentation of $QH^*_{\bbT}(T^*\calB)$, as well as an explicit description of the classical trigonometric Calogero--Moser system. 
In Section \ref{sec:ABCD}, we carry out computations of the ring presentation of $QH^*_{\bbT}(T^*\calB)$ for all classical Lie types. 
Finally, in Appendix \ref{sect:Todalimit} and Appendix \ref{sec:Acomb}, we discuss the ring presentation of $QH^*_T(\calB)$ obtained by taking the Toda limit and provide some details for the computations of $QH^*_{\bbT}(T^*\calB)$ in type $A_{n-1}$. 

\subsection*{Acknowledgement}
The authors would like to thank 
Pavel Etingof, Michael Finkelberg, Michael McBreen, Leonardo C. Mihalcea, Andrei Okounkov, Nicholas Proudfoot, Peng Shan, Kirill Zainouline, and Yehao Zhou for helpful discussions and valuable comments.
C. Li is partially supported by NSFC Grants 12271529 and 11831017. C. Su is partially supported {by a startup grant from YMSC, Tsinghua University.}
R. Xiong is partially supported by the NSERC Discovery grant RGPIN-2015-04469, Canada.

\section{Preliminaries}\label{sec:Pre}
In this section, we review some basic results about the quantum cohomology of $T^*\calB$.

\subsection*{Notations}
Let $G$ be a complex, semi-simple, simply connected group with Lie algebra $\mathfrak{g}$, and $T\subset G$ be a maximal torus with Lie algebra $\mathfrak{t}$. Let $R\subset \mathfrak{t}^*$ be the root system of $(\mathfrak{g}, \mathfrak{t})$. 
Fix a base $\Pi=\{\alpha_i\}_{i=1}^r\subset R$ of simple roots, and let $\{\omega_i\}_{i=1}^r\subset \mathfrak{t}^*$ denote the fundamental weights where $r$ is the rank of $G$. 
Let $R^\vee=\{\alpha^\vee\mid \alpha\in R\}\subset \mathfrak{t}$ denote the set of coroots. Since $G$ is simply connected, we can  canonically identify the (co)character lattice with the weight (resp. coroot) lattice:
$$X^*(T)=\operatorname{Hom}(T, \mathbb{C}^*)\cong \operatorname{Span}_{\mathbb{Z}}\{\omega_i\}_{i=1}^r,\quad \operatorname{Hom}(\mathbb{C}^*, T)\cong \operatorname{Span}_{\mathbb{Z}}R^\vee.$$

Unless otherwise stated, by $\alpha>0$ we will always mean an element $\alpha$ in the set $R^+$ of positive roots.
The associated reflection $s_\alpha: \mathfrak{t}^*\to\mathfrak{t}^*$ is defined by $s_\alpha(\lambda)=\lambda-\langle \lambda, \alpha^\vee\rangle \alpha$, where $\langle \cdot, \cdot\rangle: \mathfrak{t}^*\times \mathfrak{t}\to \mathbb{C}$ is the natural pairing. We will also denote by $s_\alpha$ the induced reflection on $\mathfrak{t}$.
The Weyl group $W$ is generated by simple reflections $s_i=s_{\alpha_i}$, for $i=1,\ldots,r$. It is equipped with the Bruhat order $\geq$, and admits a standard length function $\ell: W\to \mathbb{Z}_{\geq 0}$. 
Let $X^*(T)_+$ denote the set of dominant weights in $X^*(T)$. 
For any $\pm\lambda\in X^*(T)_+$ we denote by $W_\lambda$ the subgroup of $W$ that stabilizes $\lambda$, which is generated by simple reflections $s_i$ with $\langle \lambda, \alpha_i^\vee\rangle=0$. Each coset $uW_\lambda$ in $W/W_\lambda$ has a unique minimal length representative, denoted as $\bar u$.
Therefore $W/W_\lambda$ can be identified with the subset $W^\lambda\subset W$ of minimal length representatives. 

The \emph{degenerate affine Hecke algebra} $\calH_\hbar$ is a graded $\mathbb{C}[\hbar]$-algebra generated by $\{x_\lambda\}_{\lambda\in \ft^*}$ and $W$ such that
\begin{enumerate}
\item 
  $x_{c\lambda+\mu}=c x_\lambda+x_{\mu}$, $x_\lambda x_\mu=x_\mu x_\lambda$, $\forall\lambda,\mu\in \ft^*, \forall c\in \mathbb{C}$;
\item
	the elements $w\in W$ form the Weyl group inside $\calH_\hbar$.
\item 
  $s_ix_\lambda-x_{s_i(\lambda)}s_i=\hbar\langle\alpha_i^\vee,\lambda\rangle$, $\forall \lambda\in \ft^*, \forall i =1,\ldots,r$;
\end{enumerate}
The grading on $\calH_\hbar$ is defined by
 $\deg x_\lambda=2, \deg w=0, \deg \hbar=2$.

 Let $T^\vee:=\bbC^*\otimes_{\bbZ}X^*(T)$ be the complex dual torus, and $T^\vee_{\reg}$ be the complement of all the root hyperplanes $\{q^{\alpha^\vee} = 1\}_{\alpha\in R^+}$. Here $q^{\alpha^\vee}$ is viewed as a function over $T^\vee$, and $q^{\alpha^\vee}=\prod_i (q^{\alpha^\vee_i})^{a_i}$ for $\alpha^\vee=\sum_{i}a_i\alpha^\vee_i$. In the present paper, we will mainly take coefficients in
\begin{equation}\label{eq:defOTreg}
\calO(T^\vee_{\reg})=\mathbb{C}\big[\,q^{\alpha^\vee}, \tfrac{1}{1- q^{\alpha^\vee}}\,\big|\,\alpha\in R\,\big].
\end{equation}
For any weight $\lambda\in X^*(T)$,  we will consider the linear function ${p}_\lambda$ in $\mathcal{O}(\mathfrak{t})$
and the differential operator $\partial_\lambda$ on $\calO(T^\vee_{\reg})$,
defined respectively by
\begin{equation}\label{eqn:partial}
{p}_\lambda(\alpha^\vee)=\langle \lambda, \alpha^\vee\rangle,\qquad \partial_{\lambda}q^{\alpha^\vee}=\langle\lambda,\alpha^\vee\rangle q^{\alpha^\vee}.
\end{equation}
 We will also consider actions by the groups
 \begin{equation}
   \mathbb{G}:=G\times \mathbb{C}^*,\quad \mathbb{T}:=T\times \mathbb{C}^*.
 \end{equation}

\subsection{Equivariant cohomology of Springer resolution}
\label{sec: springer}

The complete flag variety $\calB$ of $G$ parameterizes Borel subalgebras $\mathfrak{b}$ of $\mathfrak{g}$.
 Let $\mathcal{N}$ denote the cone of nilpotent elements in $\mathfrak{g}$, i.e. those $x\in \mathfrak{g}$ with $\mbox{ad}_x\in \mbox{End}(\mathfrak{g})$ being nilpotent.
The cotangent bundle $T^*\mathcal{B}$ is isomorphic to the subvariety
 $$\tilde{\mathcal{N}}=\{(A, \mathfrak{b})\mid A\in \mathcal{N}, A\in \mathfrak{b}\}\subset \mathcal{N}\times \mathcal{B}$$
 via the identification $T^*_{\mathfrak{b}}\calB \cong [\mathfrak{b}, \mathfrak{b}]=\mathfrak{n}$. The natural projection $\tilde{\mathcal{N}}\to \mathcal{N}$, or equivalently the composition
$T^*\mathcal{B}\mathbin{\cong}\tilde{\mathcal{N}}\to \mathcal{N}$, is a resolution of singularities, called the \emph{Springer resolution}.

Let $B$ denote the standard Borel subgroup of $G$, whose Lie algebra is spanned by $\mathfrak{t}$ and those root spaces parameterized by positive roots. With the identification $\mathcal{B}=G/B$, the cotangent bundle $T^*\mathcal{B}$ admits a natural $\mathbb{G}$-action with $G$-action on $\calB$ by left multiplication and
$\bbC^*$-action given by dilation. Precisely, for any $(g,z)\in \mathbb{G}$, the action is given by
$$(g,z)\cdot (xB,\xi)
=(gxB,z^{-1}g^*\xi),$$
where $\xi$ is a cotangent vector at $gB\in G/B$ and $g^* \xi$ is pull back of it to cotangent vector at $gxB$ induced by the left translation of $g$.
 
Let $Y$ denote $T^*\calB, \calB$, or a point $\mbox{pt}$, and  $A\leq \mathbb{G}$ be a complex reductive subgroup.  Let us describe the $A$-equivariant cohomology $H_A^*(Y)=H_A^*(Y, \bbC)$ of  $Y$.
Denote also by $\hbar$  the standard generator of
 $H_{\mathbb{C}^*}^*(\mbox{pt})$.
We have
\begin{align*}
H_{T}^*(\pt) & = \Sym(\ft^*),&
H_{\mathbb{T}}^*(\pt) & = \Sym(\ft^*)[\hbar],\\
H_{G}^*(\pt) & = \Sym(\ft^*)^W,&
H_{\mathbb{G}}^*(\pt) & = \Sym(\ft^*)^W[\hbar].
\end{align*}
Note that $T^*\calB$ is a vector bundle over $\calB$, thus
\begin{align}
H^*_{\mathbb{G}}(T^*\calB)
& =H^*_{\mathbb{G}}(\calB)
=H^*_{G}(\calB)[\hbar],\label{eq:HGGBh}\\
\text{and \quad }H^*_{\mathbb{T}}(T^*\calB)
& =H^*_{\mathbb{T}}(\calB)
=H^*_{T}(\calB)[\hbar].\label{eq:HTGBh}
\end{align}
Every $\lambda\in X^*(T)$ defines a line bundle   $\calL_\lambda:=G\times_B\bbC_\lambda$   over $\calB$. The equivariant first Chern classes of $\calL_\lambda$ are denoted as: 
\begin{equation}
  D_{\lambda}^{G, \calB} = c_1^{G}(\mathcal{L}_{\lambda})\in H_G^2(\calB), \quad D_{\lambda}^{T, \calB} = c_1^{T}(\mathcal{L}_{\lambda}) \in H_T^2(\calB).
\end{equation}
Similarly, for $A\leq \mathbb{G}$, we denote by $D_{\lambda}^{A, T^*\calB}$ the $A$-equivariant first Chern class of the line bundle over $T^*\calB$ obtained via pullback of $\calL_\lambda$. Whenever there is no confusion, we will simply denote all classes as $D_\lambda$.

Let $\St:=T^*\calB\times_\calN T^*\calB$ denote the \emph{Steinberg variety}, which inherits a $\mathbb{G}$ action.
Let $H_*^{\mathbb{G}}(\St)$ denote the $\mathbb{G}$-equivariant Borel--Moore homology of the Steinberg variety, which has an associative algebra structure via convolution, see \cite{MR2838836}.
The equivariant cohomology $H^*_{\mathbb{G}}(T^*\calB)$ admits a canonical action by $H_*^{\mathbb{G}}(\St)$, and hence by the affine Hecke algebra $\calH_\hbar$, due to an isomorphism $\calH_\hbar\simeq H_*^{\mathbb{G}}(\St)$ of 
 $\bbC[\hbar]$-algebras proved by Lusztig \cite{MR972345}. 
The isomorphism can be explicitly described as follows.
\begin{enumerate} 
    \item For any $\lambda\in X^*(T)$, $x_{\lambda}\in \calH_\hbar$ is sent to the the diagonal push forward of $D_{\lambda}$ from $T^*\calB$ to $\St$.
In particular, for any $\gamma\in H_{\mathbb{G}}(T^*\calB)$, the action of $x_{\lambda}$ on $\gamma$ is given by the equivariant product with $D_{\lambda}$, i.e.
$$x_{\lambda}(\gamma) =D_{\lambda}\cdot \gamma.$$
    \item Let $\calP_i:=G/P_i$, where $P_i:=B\cup Bs_iB$ is the minimal parabolic subgroup containing $B$. 
The fundamental class of the conormal bundle of $\calB\times_{\calP_i}\calB$ inside $\calB\times \calB$ is a well defined element in $H_*^{\mathbb{G}}(\St)$, and it corresponds to $s_i-1$. 
\end{enumerate}
 The affine Hecke algebra $\calH_\hbar$ also acts on the $\mathbb{T}$-equivariant cohomology 
  $H^*_{\mathbb{T}}(T^*\calB)=H^*_{\mathbb{T}}(\pt)\otimes_{H^*_{\mathbb{G}
  }(\pt)}H^*_{\mathbb{G}}(T^*\calB)$.
 The operators $x_\lambda$, $w$ are all $H^*_{\mathbb{T}}(\pt)$-linear. Moreover, the operators $w\in W\subset \calH_\hbar$ are also referred to as 
 the  Demazure--Lusztig operators, see \cite{MR1649626}.

Due to Borel, there is an isomorphism, 
\begin{equation}\label{eq:BorelHGGB}
\calO(\ft)\stackrel{\sim}\longrightarrow H_G^*(\calB)
\end{equation}
which sends ${p}_{\lambda}$ to $D_{\lambda}$; moreover, the following holds.
\begin{Prop}[\cite{MR51508}, see also \cite{AFbook}]\label{Th:BorelTh}
We have the following ring presentation
\begin{align*}
H_T^*(\calB)
& =\mathbb{C}[\lambda,D_{\lambda}]_{\lambda\in X(T)}\big/\big<f(D_{\lambda})-f(\lambda)\mid f(\lambda)\in \Sym(\ft^*)^W\big>,
\end{align*}
 where $f(D_{\lambda})$ is obtained by substituting $\lambda$ by $D_{\lambda}$. 
\end{Prop}

\subsection{Calogero--Moser system}\label{sec:cm}
The Calogero--Moser system is a Hamiltonian system in the study of one-dimensional many-body problems. The system was originally written down in its rational version and was generalized to the trigonometric case and the elliptic case by Sutherland and Krichever, respectively (see e.g. \cite{MR2296754,MR3817553} and references therein).

For our purpose, we need the \emph{trigonometric Calogero--Moser system} mostly and it is an integrable system over $\calO(T^\vee_{\reg}\times \ft)[\hbar]$ with Hamiltonian defined as follows. For any non-degenerate $W$-invariant quadratic form $C\in \Sym^2(\ft^*)^W$, we define the Hamiltonian to be
\begin{equation}\label{equ:etaCcl}
C(p)-\hbar^2
\sum_{\alpha>0}\frac{C(\alpha^\vee,\alpha^\vee)}{(q^{\alpha^\vee/2}-q^{-\alpha^\vee/2})^2}
\in \calO(T^\vee_{\reg}\times \ft)[\hbar],
\end{equation}
where ${p}_{\lambda}$ is defined \eqref{eqn:partial} and naturally viewed as a function in $\calO(T^\vee_{\reg}\times \ft)$. 
The trigonometric Calogero--Moser system can be described by a $\bbC[\hbar]$-linear embedding
\begin{equation}\label{eq:gaugeMCtoQH}
  \eta_{\CM,\hbar}^{\cl}:\Sym(\ft^*)^W[\hbar]\hookrightarrow \calO(T^\vee_{\reg}\times \ft)[\hbar],
\end{equation}
whose image consists of Possion commuting functions characterized by
\begin{enumerate} 
\item for any $f\in \Sym(\ft^*)^W$, $f(p)$ appears as the leading term of $\eta^{\cl}_{\CM,\hbar}(f)$;
\item for any $C\in \Sym^2(\ft^*)^W $,
$\eta_{\CM,\hbar}^{\cl}(C)$ is the Hamiltonian defined in \eqref{equ:etaCcl}.
\end{enumerate}
The existence of such a map can be ensured by the construction of Dunkl operators. 
For any $\lambda\in X^*(T)$, the \emph{classical trigonometric Dunkl operator} is
\begin{align}\label{eq:cltriDunkop}
\Dun_{\lambda}&:=p_{\lambda}-\hbar\sum_{\alpha>0}\left<\lambda,\alpha^\vee\right>\frac{q^{\alpha^\vee}}{1-q^{\alpha^\vee}}s_\alpha \in \calO(T^\vee_{\reg}\times \ft)[\hbar]\rtimes \bbC[W].
\end{align}
For any polynomial $f(\lambda)\in \Sym(\ft^*)$, writing 
\[f(\Dun_{\lambda})=\sum_{w\in W} D(w,f,\hbar)\cdot w \in \calO(T^\vee_{\reg}\times \ft)[\hbar] \rtimes \mathbb{C}[W],\]
where $D(w,f,\hbar)\in \calO(T^\vee_{\reg}\times \ft)[\hbar]$, 
we define 
\[D(f,\hbar):=\sum_w D(w,f,\hbar)\in \calO(T^\vee_{\reg}\times \ft)[\hbar].\]
Then $\eta_{\CM,\hbar}^{\cl}(f) = D(f,\hbar)$ for $f\in \Sym(\ft^*)^W$.

Let $\calD(T^\vee_{\reg})$ denote the $\mathbb{C}$-algebra of differential operators on $T^\vee_{\reg}$.
The quantum version of the trigonometric Calogero--Moser system can be similarly described by 
\begin{align*}
\eta_{\CM,\hbar}:\Sym(\ft^*)^W[\hbar]\hookrightarrow\calD(T^\vee_{\reg})[\hbar]
\end{align*}
characterized by similar conditions and can be constructed similarly using (quantum) trigonometric Dunkl operators \cite{MR1627113,MR2769318}. 

For our purpose, we need a gauge transformation.
Let 
 \begin{equation}
   \delta=\prod_{\alpha>0}(q^{\alpha^\vee}-1).
 \end{equation}
This defines an automorphism $\vartheta$ over $\calD(T^\vee_{\reg})[\hbar]$ given by $\vartheta(x)=\delta^{\hbar}x\delta^{-\hbar}$.
By direct computation, we have
\begin{align*}
\vartheta(q^{\alpha^\vee})= q^{\alpha^\vee}
\quad\text{and}\quad
\vartheta(\partial_{\lambda})
= \partial_\lambda+\hbar\sum_{\alpha>0}
\frac{\left<\lambda,\alpha^\vee\right>q^{\alpha^\vee}}{1-q^{\alpha^\vee}}.
\end{align*}
This induces an automorphism $\vartheta^{\cl}$ of the Poisson algebra $\calO(T^\vee_{\reg}\times \ft)[\hbar]$ such that
\begin{align*}
\vartheta^{\cl}(q^{\alpha^\vee})= q^{\alpha^\vee}
\quad\text{and}\quad
\vartheta^{\cl}({p}_\lambda) = {p}_\lambda+
\hbar\sum_{\alpha>0}
\frac{\left<\lambda,\alpha^\vee\right>q^{\alpha^\vee}}{1-q^{\alpha^\vee}}.
\end{align*}
Let us denote
\begin{align}\label{eq:gaugeMCtoQD} 
\overline{\eta}_{\CM,\hbar}&=
\vartheta\circ \eta_{\CM,\hbar}, \qquad
\overline{\eta}^{\cl}_{\CM,\hbar} =
\vartheta^{\cl}\circ \eta^{\cl}_{\CM,\hbar}.
\end{align}

\subsection{Quantum cohomology}
We will study the $A$-equivariant quantum cohomology $QH_A^*(X)$ of a space $X\in \{\calB, T^*\calB\}$. Here $A\leq \mathbb{G}$ is a complex reductive subgroup. In particular for $X=T^*\calB$, we require that $A$ contains 
$\{e\}\times \mathbb{C}^*$, so that the $A$-fixed locus $X^A$ is always compact. 

Note that $G$ is simply connected. We have the identifications
\[\mbox{Span}_{\mathbb{Z}}R^\vee=\Hom(\bbC^*,T)\simeq H_2(X,\bbZ).\]
Every positive coroot $\alpha^\vee$ defines an $\SL_2$-subgroup $G_{\alpha^\vee}\subset G$, and then defines a rational curve $G_{\alpha^\vee} B/B\subset \calB\subset T^*\calB$. Therefore the cone $\Eff(X)$ of effective classes in $H_2(X,\bbZ)$ is generated by the positive coroots. 
For each effective curve class $\beta\in \Eff(X)\subset H_2(X,\bbZ)$, and $\gamma_1,\ldots,\gamma_k\in H_{A}(X)$, the $k$-point, genus $0$, equivariant \emph{Gromov--Witten invariants} are given by
\begin{equation*}
\left<\gamma_1,\ldots,\gamma_k\right>_{0,k,\beta}^X=\int_{[\overline{M}_{0,k}(X,\beta)]^{\mathrm{vir}}} \ev^*(\gamma_1\boxtimes \cdots \boxtimes \gamma_k),
\end{equation*}
where $[\overline{M}_{0,k}(X,\beta)]^{\mathrm{vir}}$ is the virtual fundamental class of the moduli space of $k$-pointed stable maps to $X$, see \cite{MR1677117} and \cite[section 4]{MR2782198}.
For $X=T^*\calB$ which is noncompact, the integral is defined via equivariant residue, see \cite{MR3184181}.
The \emph{quantum cohomology} 
$$QH_{A}^*(X)=(H_{A}^*(X)\otimes_{\bbC}\bbC
[\![q^{\beta}]\!]_{\beta\in \Eff(X)}, \,*\,)$$
is a deformation of the equivariant cohomology $H_A^*(X)$.
The quantum product $*$ is a $\mathbb{C}[\![q^{\beta}]\!]_{\beta\in \Eff(X)}$-bilinear product over $QH_{A}^*(X)$ such that
for any $\gamma_1,\gamma_2,\gamma_3\in H_{A}^*(X)$, we have
\[(\gamma_1*\gamma_2,\gamma_3)=\sum_{\beta\in \Eff(X)\cup\{0\}}q^\beta\langle \gamma_1,\gamma_2,\gamma_3\rangle_{0,3,\beta}^X.\]
Here $(\cdot,\cdot)$ is the Poincar\'e pairing, which can be evaluated using equivariant residue in the case of $T^*\calB$.

It is well-known that $QH_{A}^*(X)$ forms a graded commutative $H_{A}^*(\pt)$-algebra with identity $1$.
Here, the degree of $q^\beta$ is defined to be
$$\deg q^{\beta} = 2\left<-K_X,\beta\right>,$$
where $K_X$ is the canonical bundle of $X$. In particular,
the degree of $q^{\alpha^\vee}$ is zero in $QH^*_{A}(T^*\calB)$, since $K_{T^*\calB}$ is trivial.

 On the trivial vector bundle on $H^2(X)$ with fiber $H_A^*(X)$, there is a flat connection $\nabla$, called the \emph{Dubrovin connection} and defined by 
\[\nabla_\lambda:=\partial_{\lambda}-\lambda *, \quad\forall \lambda\in H^2(X),\]
where $\partial_{\lambda}$ is as given in \eqref{eqn:partial}, by identifying $H^2(X)$ with $\mathfrak{t}^*$.  The equivariant cohomology $QH_{A}^*(X)$ is then equipped with an $H_{A}^*(\pt)$-$\calD$-module structure, known as the \emph{$A$-equivariant quantum $\calD$-module} of $X$. The main results of \cite{MR2782198} are the following.

\begin{Prop}[{\cite[Theorem 3.2]{MR2782198}}]\label{thm:BMO}
\mbox{}
\begin{enumerate}[\quad\rm(1)]
\item \label{thm:BMO(1)}
  The quantum multiplication by $D_\lambda$ is given by
  \[x_\lambda+\hbar\sum_{\alpha>0}\left<\lambda,\alpha^\vee\right>\frac{q^{\alpha^\vee}}{1-q^{\alpha^\vee}}(s_\alpha-1).\]
\item
  The quantum connection of $T^*\calB$ can be analytically extended to $T^\vee_{\reg}$ and is  given by
  \[\nabla_\lambda=\partial_{\lambda}-x_\lambda-\hbar\sum_{\alpha>0}\left<\lambda,\alpha^\vee\right>\frac{q^{\alpha^\vee}}{1-q^{\alpha^\vee}}(s_\alpha-1).\]
  Moreover, the $H_{\mathbb{G}}^*(\pt)$-structure is given by the quantum Calogero--Moser map
  $\overline{\eta}_{\CM,\hbar}$. 
\item \label{thm:BMO(3)}
  We have an $\calO(T^\vee_{\reg})$-algebra isomorphism 
  \begin{equation}\label{eq:BMOisoQHGTGB}
  \calO(T^\vee_{\reg}\times \ft)[\hbar]\stackrel{\sim}\longrightarrow QH_{\mathbb{G}}^*(T^*\calB)
  \end{equation}
  sending ${p}_{\lambda}$ to $D_{\lambda}$.
  Moreover, the $H_{\mathbb{G}}^*(\pt)$-structure is given by the classical Calogero--Moser map
  $\overline{\eta}^{\cl}_{\CM,\hbar}$. 
\end{enumerate}
\end{Prop}

\begin{Rmk}
Part (\ref{thm:BMO(3)}) of the above proposition admits an equivalent version as follows. By introducing
\begin{equation}\label{eq:Deltacls}
\Delta_{\lambda} = D_{\lambda}+\hbar\sum_{\alpha>0}\left<\lambda,\alpha^\vee\right>\frac{q^{\alpha^\vee}}{1-q^{\alpha^\vee}},
\end{equation}
we have an $\calO(T^\vee_{\reg})$-algebra isomorphism as \eqref{eq:BMOisoQHGTGB}, which sends ${p}_{\lambda}$ to $\Delta_{\lambda}$. In this case, the $H_{\mathbb{G}}^*(\pt)$-structure is given by the classical Calogero--Moser map ${\eta}^{\cl}_{\CM,\hbar}$ (without gauge transformation). 
The relation is summarized in the following diagram
$$\xymatrix{
&&\mathcal{O}(T^\vee_{\reg}\times \ft^*)[\hbar]
\ar[dr]^{{p}_{\lambda}\mapsto D_{\lambda}}\\
H^*_{\mathbb{G}}(\pt) \ar[r]^-{\sim}& \Sym(\ft^*)^W[\hbar]
\ar[ur]^{\overline{\eta}_{\CM,\hbar}^{\cl}}
\ar[dr]_{\eta_{\CM,\hbar}^{\cl}}
&&QH^*_{\mathbb{G}}(T^*\calB)\\
&&\mathcal{O}(T^\vee_{\reg}\times \ft^*)[\hbar]
\ar[ur]_{{p}_{\lambda}\mapsto \Delta_{\lambda}}
\ar[uu]_{\vartheta^{\cl}}
}
$$
\end{Rmk}

\begin{Rmk}
Note that Proposition \ref{thm:BMO} (\ref{thm:BMO(3)}) shows that $QH^*_{\mathbb{G}}(T^*\calB)$
is generated by divisors $D_\lambda$ (or equivalently classes $\Delta_{\lambda}$). 
The isomorphism of \eqref{eq:BMOisoQHGTGB} is obtained after refining the equivariant quantum cohomology
\begin{equation}	
QH^*_{\mathbb{G}}(T^*\calB)=H^*_{\mathbb{G}}(T^*\calB)\otimes_{\bbC}\mathcal{O}(T^\vee_{\reg})
\end{equation}
by replacing the formal power series $\bbC[\![q^\beta]\!]_{\beta}$ with the function ring $\calO(T^\vee_{\reg})$.
In other words, the quantum product can be analytically extended to the ring $T^\vee_{\reg}$. 

The same argument also applies to $QH^*_{\mathbb{T}}(T^*\calB)$. That is, $QH^*_{\mathbb{T}}(T^*\calB)$
is generated by divisors $D_\lambda$ (or equivalently $\Delta_{\lambda}$) and we can realize it over
\begin{equation}	
QH^*_{\bbT}(T^*\calB)=H^*_{\bbT}(T^*\calB)\otimes_{\bbC}\mathcal{O}(T^\vee_{\reg}).
\end{equation}
These facts are crucial when defining Quantum Demazure--Lusztig operators.
\end{Rmk}

\section{Automorhisms of {$QH^*_{\mathbb{T}}(T^*\calB)$}}\label{sec:Aut}

Let us define the main objects of this paper. 

\begin{Def}[Quantum Demazure--Lusztig operators]
\label{defofQDLop}
For any $u\in W$, we define the \emph{quantum Demazure--Lusztig operator} $T_u$ to be the 
operator on $QH^*_{\mathbb{T}}(T^*\calB)=
H^*_{\mathbb{T}}(T^*\calB)\otimes \calO(T^\vee_{\reg})$ given by
\begin{equation}\label{equ:DL}
T_u= u\otimes u,
\end{equation}
where the first $u$ on the right-hand-side is an element of $\calH_{\hbar}$, i.e. a usual Demazure--Lusztig operator, 
and the second $u$ denotes the induced action on $\calO(T^\vee_{\reg})$ from that on $T^\vee_{\reg}$, namely 
 $u(q^{\alpha^\vee}):=q^{u\alpha^\vee}$. 
\end{Def}

This section is devoted to the following surprising theorem, which will play an important role in finding the presentation of $QH^*_{\mathbb{T}}(T^*\calB)$ in Section \ref{sec:Pres}. 

\begin{Th}\label{Swisring}
For any $u\in W$ and $\gamma_1,\gamma_2\in QH_{\bbT}^*(T^*\calB)$,
\begin{equation}\label{eq:Swisring}
T_u(\gamma_1 * \gamma_2) = T_u(\gamma_1)*T_u(\gamma_2).
\end{equation}
As a result, $T_u$ is a ring automorphism with respect to the quantum product.
\end{Th}

\begin{Eg}\label{eg:autoP1}Consider the case $G=SL(2,\bbC)$. 
Let $s:=s_\alpha\in W$ be the only simple reflection, $\varpi$ be the fundamental weight, and $q:=q^{\alpha^\vee}$ denote the unique primitive effective curve class. Note that $-D_{\varpi}=D_{-\varpi}\in H_{\bbT}^*(T^*\bbP^1)$ is the pullback of the equivariant first Chern class $c_1^{\bbT}(\calO(1))\in H_{\bbT}^*(\bbP^1)$.
By Proposition \ref{thm:BMO} (\ref{thm:BMO(1)}), 
\begin{equation*}
D_{\varpi}*D_{-\varpi} 
= D_{\varpi}D_{-\varpi}+\hbar\frac{q}{1-q} (D_{\varpi}-D_{-\varpi}-\hbar).
\end{equation*}
We have
\begin{align*}
T_s(D_{\varpi})*T_s(D_{-\varpi}) 
& = (D_{-\varpi}+\hbar)*(D_{\varpi}-\hbar)\\
& = 
D_{\varpi}D_{-\varpi}+\hbar\frac{1}{1-q}
(D_{\varpi}-D_{-\varpi}-\hbar),\\
T_s(D_{\varpi}*D_{-\varpi}) & = D_{\varpi}D_{-\varpi}+\hbar\frac{q^{-1}}{1-q^{-1}}(D_{-\varpi}-D_{\varpi}+\hbar)\\
& = D_{\varpi}D_{-\varpi}+\hbar\frac{1}{1-q}(-D_{-\varpi}+D_{\varpi}-\hbar).
\end{align*}
Thus 
$$T_s(D_{\varpi}*D_{-\varpi})
=T_s(D_{\varpi})*T_s(D_{-\varpi}).$$
\end{Eg}

\subsection{Stable basis}
The \emph{stable basis} was introduced by Maulik and Okounkov in their seminal work on quantum cohomology of Nakajima quiver varieties \cite{MR3951025}. It plays a crucial role in our proof of Theorem \ref{Swisring}.
Let us briefly recall the definition of the stable basis for the Springer resolution, see
\cite[Chapter~3]{MR3951025} or \cite{MR3595900} for more details.

Recall that the torus $\mathbb{T}$ acts on $T^*\calB$, and the
fixed points are $(wB,0)\in T^*\calB$ for $w\in W$. For any
$w\in W$ and $\gamma\in H_{\bbT}^*(T^*\calB)$, we denote by
$\gamma|_w$ the restriction of $\gamma$ to the fixed point $(wB,0)$.
The definition of stable envelopes depends on a choice of a Weyl
chamber in $\ft$.
For our purpose, we need to choose the negative Weyl chamber, for which we use $-$ in the subscript to indicate.
Let $B^-$ denote the opposite Borel subgroup.
The stable basis is characterized by the following proposition.
\begin{Prop}[\cite{MR3951025,MR3595900}]
\label{rem:stab}
 There exists a unique family of classes
 \[ \big\{\Stab_-(w)\in H_{\bbT}^{2 \dim \calB}(T^*\calB) \,\big|\, w\in
 W\big\} \] which satisfies the following properties:
\begin{enumerate}[\quad\rm(1)]
\item \label{def:stab(1)}
$\Stab_-(w)$ is supported on the union of conormal bundle of opposite Schubert varieties $B^-uB/B$ for $u\geq w$, i.e., $\Stab_-(w)|_u = 0$ unless $u \geq w$;
\item \label{def:stab(2)}
$\Stab_-(w)|_w=
 \prod\limits_{\alpha>0,w\alpha>0}(w\alpha-\hbar)\prod\limits_{\alpha>0,w\alpha<0}w\alpha$;
\item \label{def:stab(3)}
$\Stab_-(w)|_u$ is divisible by $\hbar$, for any $u>w$.
\end{enumerate}
Moreover, the stable basis forms a basis of $H_{\bbT}^*(T^*\calB)$ after inverting the equivariant parameters. 
\end{Prop}

\begin{Prop}[{\cite[Theorem 3.14]{MR3439689}}]\label{prop:quantum}
For $w\in W$ and $\lambda\in X^*(T)$, the following holds in $QH_{\bbT}^*(T^*\calB)$
  \begin{align*}
    D_\lambda*\Stab_-(w)&=w(\lambda)\Stab_-(w)-\hbar\sum_{\alpha>0,w\alpha>0}\left<\lambda,\alpha^\vee\right>\Stab_-(ws_\alpha)\\
    &\quad -\hbar\sum_{\alpha>0}\left<\lambda,\alpha^\vee\right>\frac{q^{\alpha^\vee}}{1-q^{\alpha^\vee}}\big(\Stab_-(w)+\Stab_-(ws_\alpha)\big).
  \end{align*}
\end{Prop}

\begin{Rmk}
  The above formula can also be deduced from Propositions \ref{thm:BMO} (\ref{thm:BMO(1)}) as well as 
  the Chevalley formula \cite[Theorem 3.7]{MR3439689} of the stable basis for cup product. 
\end{Rmk}

Recall the class $\Delta_{\lambda}$ defined in \eqref{eq:Deltacls}.
Proposition \ref{prop:quantum} can be rewritten as
\begin{align}
\Delta_{\lambda}*\Stab_-(w) &
= w(\lambda) \Stab_-(w)
-\hbar\sum_{\alpha>0,w\alpha>0}\left<\lambda,\alpha^\vee\right>\frac{1}{1-q^{\alpha^\vee}}\Stab_-(ws_\alpha)\nonumber\\
    &-\hbar\sum_{\alpha>0,w\alpha<0}\left<\lambda,\alpha^\vee\right>\frac{q^{\alpha^\vee}}{1-q^{\alpha^\vee}}\Stab_-(ws_\alpha).\label{eq:qChe}
\end{align}

\subsection{Proof of Theorem \ref{Swisring}}
Let us first prove the following lemma.

\begin{Lemma}\label{lem:TDel}
For any $w,u\in W$ and $\lambda\in X^*(T)$, we have
\begin{enumerate}[\quad \rm(1)]
\item \label{lem:TDel(1)}
$T_u\big(\Stab_-(w)\big) =(-1)^{\ell(u)}\Stab_-(wu^{-1})$; 
\item \label{lem:TDel(2)}
$T_u (\Delta_\lambda) = \Delta_{u\lambda}$.
\end{enumerate}
\end{Lemma}
\begin{proof}
Since $\Stab_-(w)$ has no quantum component, we have $T_u(\Stab_-(w))=u(\Stab_-(w))$. 
We know from \cite[Lemma 3.2]{MR3595900} that 
$u\,\big(\Stab_-(w)\big)=(-1)^{\ell(u)}\Stab_-(wu^{-1})$. 
This proves the first part (\ref{lem:TDel(1)}). 

For the second part (\ref{lem:TDel(2)}), it suffices to show the case when $w=s_i$ is a simple reflection:
\begin{align*}
T_{s_i}(\Delta_{\lambda})
& = s_i(D_{\lambda}) + \hbar\sum_{\alpha>0} \left<\lambda,\alpha^\vee\right>\frac{q^{s_i\alpha^\vee}}{1-q^{s_i\alpha^\vee}}\\
& = D_{s_i\lambda}+\hbar\left<\lambda,\alpha_i^\vee\right> + \hbar\sum_{\alpha>0,\alpha\neq \alpha_i} \left<s_i\lambda,\alpha^\vee\right>\frac{q^{\alpha^\vee}}{1-q^{\alpha^\vee}}+\hbar\left<\lambda,\alpha_i^\vee\right>\frac{q^{-\alpha_i^\vee}}{1-q^{-\alpha_i^\vee}}\\
& = \Delta_{s_i\lambda}.
\end{align*}
Here in the second equality, we have used the facts $s_i(R^+\setminus\{\alpha_i\})=R^+\setminus\{\alpha_i\}$ and 
$s_i(D_{\lambda})=s_ix_\lambda \cdot 1=(x_{s_i\lambda} +\hbar\langle\lambda,\alpha_i^\vee\rangle)\cdot 1= D_{s_i\lambda}+\hbar\langle\lambda,\alpha_i^\vee\rangle$.
\end{proof}

With this lemma, we can prove the following special case of Theorem \ref{Swisring}.
\begin{Prop}\label{SDeltaStab}
For any $w,u\in W$ and $\lambda\in X^*(T)$,
\begin{equation}
T_u(\Delta_{\lambda}*\Stab_-(w)) =
T_u(\Delta_{\lambda}) * T_u(\Stab_-(w)).
\end{equation}
\end{Prop}
\begin{proof}The proof is given by direct computation using the definition of $T_u$, Lemma \ref{lem:TDel}, and \eqref{eq:qChe}. On one hand, we have
\begin{align*}
  &\quad (-1)^{\ell(u)}T_u(\Delta_{\lambda}) * T_u(\Stab_-(w))=\Delta_{u\lambda}*\Stab_-(wu^{-1})\\
  &=w(\lambda) \Stab_-(wu^{-1}) -\hbar\sum_{\alpha>0,wu^{-1}\alpha>0}\left<u\lambda,\alpha^\vee\right>\frac{1}{1-q^{\alpha^\vee}}\Stab_-(wu^{-1}s_\alpha)\\
    &\quad -\hbar\sum_{\alpha>0,wu^{-1}\alpha<0}\left<u\lambda,\alpha^\vee\right>\frac{q^{\alpha^\vee}}{1-q^{\alpha^\vee}}\Stab_-(wu^{-1}s_\alpha).
\end{align*}
On the other hand, we have
\begin{align*}
  &\quad (-1)^{\ell(u)}T_u(\Delta_{\lambda}*\Stab_-(w)) \\
  &= w(\lambda) \Stab_-(wu^{-1})-\hbar\sum_{\beta>0,w\beta>0}
\left<\lambda,\beta^\vee\right>\frac{1}{1-q^{u\beta^\vee}}\Stab_-(ws_\beta u^{-1})\\
    &\quad-\hbar\sum_{\beta>0,w\beta<0}\left<\lambda,\beta^\vee\right>\frac{q^{u\beta^\vee}}{1-q^{u\beta^\vee}}\Stab_-(ws_\beta u^{-1})
\end{align*}
We finish the proof by matching the coefficients in the two expressions above.

\medbreak
{\bf Case A.}
The coefficients of $\Stab_-(wu^{-1})$ are both $w(\lambda)$.

\medbreak
{\bf Case B.}
Let $\alpha>0$ such that $u^{-1}\alpha=\beta>0$. 
If $wu^{-1}\alpha>0$, i.e. $w\beta >0$, we have
\begin{align*}
\left<u\lambda,\alpha^\vee\right>\dfrac{1}{1-q^{\alpha^\vee}}\Stab_-(wu^{-1}s_{\alpha})
& = 
\left<\lambda,\beta^\vee\right>\dfrac{1}{1-q^{u\beta^\vee}}\Stab_-(ws_{\beta}u^{-1}).
\end{align*}
Similarly, if $wu^{-1}\alpha<0$, i.e. $w\beta <0$, we have
\begin{align*}
\left<u\lambda,\alpha^\vee\right>\dfrac{q^{\alpha^\vee}}{1-q^{\alpha^\vee}}\Stab_-(wu^{-1}s_{\alpha})
& = 
\left<\lambda,\beta^\vee\right>\dfrac{q^{u\beta^\vee}}{1-q^{u\beta^\vee}}\Stab_-(ws_{\beta}u^{-1}).
\end{align*}

\medbreak
{\bf Case C.}
Let $\alpha>0$ such that $u^{-1}\alpha=-\beta>0$. 
If $wu^{-1}\alpha>0$, i.e. $w\beta < 0$, we have
\begin{align*}
\left<u\lambda,\alpha^\vee\right>\dfrac{1}{1-q^{\alpha^\vee}}\Stab_-(wu^{-1}s_{\alpha})
& = \left<\lambda,-\beta^\vee\right>\dfrac{1}{1-q^{-u\beta^\vee}}\Stab_-(ws_{\beta}u^{-1})\\
& = \left<\lambda,\beta^\vee\right>\dfrac{q^{u\beta^\vee}}{1-q^{u\beta^\vee}}\Stab_-(ws_{\beta}u^{-1}).
\end{align*}
Similarly, if $wu^{-1}\alpha<0$, i.e. $w\beta >0$, we have
\begin{align*}
\left<u\lambda,\alpha^\vee\right>\dfrac{q^{\alpha^\vee}}{1-q^{\alpha^\vee}}\Stab_-(wu^{-1}s_{\alpha})
& = \left<\lambda,-\beta^\vee\right>\dfrac{q^{-u\beta^\vee}}{1-q^{-u\beta^\vee}}\Stab_-(ws_{\beta}u^{-1})\\
& = \left<\lambda,\beta^\vee\right>\dfrac{1}{1-q^{u\beta^\vee}}\Stab_-(ws_{\beta}u^{-1}).\qedhere
\end{align*}
\end{proof}

We also need the following easy lemma.
\begin{Lemma}\label{lem:prep}
Let $\gamma_1,\gamma_2,\gamma_3\in QH_{\bbT}^*(T^*\calB)$.
\begin{enumerate}[\quad \rm(1)]
\item\label{lem:prep(1)}
    Assume $T_u(\gamma_1 * \gamma_2)=T_u(\gamma_1)*T_u(\gamma_2)$ and $T_u(\gamma_1 * \gamma_3)=T_u(\gamma_1)*T_u(\gamma_3)$. Then
    for any $f_2,f_3\in H^*_{\mathbb{T}}(\pt)\otimes \calO(T^\vee_{\reg})$,
    \[T_u(\gamma_1 * (f_2\gamma_2+f_3\gamma_3))=T_u(\gamma_1)*T_u(f_2\gamma_2+f_3\gamma_3).\]
\item\label{lem:prep(2)}
    Assume $T_u(\gamma_i * \gamma)=T_u(\gamma_i)*T_u(\gamma)$ for $i=1,2$ and any $\gamma\in QH_{\bbT}^*(T^*\calB)$. Then
    \[T_u(\gamma_1*\gamma_2 * \gamma)=T_u(\gamma_1*\gamma_2)*T_u(\gamma).\]
\end{enumerate}
\end{Lemma}
\begin{proof}
  The first one follows from the definition of $T_u$, while the second one follows from the associativity of the quantum product.
\end{proof}

Now we can prove Theorem \ref{Swisring}.

\begin{proof}[Proof of Theorem \ref{Swisring}] 
Recall that by Proposition \ref{rem:stab} the stable basis forms a basis of $H_{\bbT}(T^*\calB)$ after inverting the equivariant parameters. Thus by Proposition \ref{SDeltaStab} and Lemma \ref{lem:prep} (\ref{lem:prep(1)}), \eqref{eq:Swisring} is true for $\gamma_1=\Delta_{\lambda}$ with $\lambda\in X^*(T)$ and arbitrary $\gamma_2\in QH_{\bbT}(T^*\calB)$.
Since $QH_{\bbT}(T^*\calB)$ is generated by classes $\Delta_{\lambda}$, the theorem follows from Lemma \ref{lem:prep} (\ref{lem:prep(2)}). 
\end{proof}

\section{Presentation of Quantum Cohomology}\label{sec:Pres}

The equivariant cohomology of $T^*\calB$ is generated by the divisors $D_\mu$. Thus there is a natural surjective morphism $\pi$ of algebras over $\calO(T_{\reg}^\vee)[\hbar]$,
\begin{equation}\label{equ:pi}
\pi:\calO(T_{\reg}^\vee)[\hbar][\mu,D_{\mu}]_{\mu\in X^*(T)} \longrightarrow QH_{\bbT}^*(T^*\calB).
\end{equation}
In this section, we will determine the kernel $\ker \pi$. More precisely,
for any (anti-)dominant weight $\lambda\in X^*(T)_\pm$ and positive integer $k$, we let 
\begin{equation}
\mathcal{R}_{\lambda}^k=\operatorname{tr}\big(\Theta(\lambda)^k\big)-\sum_{u\in W^\lambda} (u\lambda)^k
\end{equation}
where $\Theta(\lambda)=\big(\Theta_{{u},{v}}(\lambda)\big)_{{u},{v}\in W^\lambda}$ is a square matrix of order $|W^\lambda|=|W\!\lambda|$  defined by 
 \begin{equation}\label{eq:defofThetalambda}
\Theta_{{u},{v}}(\lambda) =
\begin{cases}
\Delta_{u\lambda}, & \mbox{if } {u}={v},\\[1ex]
\frac{-\hbar\left<\lambda,\alpha^\vee\right>}{1-q^{u\alpha^\vee}}, & \text{if $\o{us_\alpha}=v$ for some $\alpha\in R^+\setminus R_\lambda^+$},\\[1ex]
0,& \text{otherwise}.
\end{cases}
\end{equation}
Here $R_\lambda^+:=R^+\bigcap (\sum_{s_i\in W_\lambda}\mathbb{Z}\alpha_i)$, and we recall that  $\o{us_\alpha}\in W^\lambda$ denotes the minimal length representative of the coset $us_\alpha W_\lambda\in W/W_\lambda$. The root $\alpha$ in the second case is unique \cite{MR2072765}. The following theorem is the main result of this section. 
\begin{Th}\label{qHpre}
We have the following presentation:
\begin{align*}
QH_{\bbT}^*(T^*\calB)
&\cong \calO(T^\vee_{\reg})[\hbar][\mu,D_{\mu}]_{\mu\in X^*(T)}\bigg/\big<\mathcal{R}_{\lambda}^k\mid\lambda\in X^*(T)_+, k> 0\big>.
\end{align*}
\end{Th}

\begin{Eg}\label{preforsl2}
Let $G=SL(2, \bbC)$, and use the same notation as in Example \ref{eg:autoP1}. 
Then
$$\Theta(\varpi)=
\bigg[\!\!\begin{subarray}{c}
  \phantom{\frac{1}{2}}
D_{\varpi}+\hbar\frac{q}{1-q}
  \phantom{\frac{1}{2}}\\
\frac{-\hbar}{1-q^{{-}1}}
\end{subarray}\!\!\!\!
\begin{subarray}{c}
\frac{-\hbar}{1-q}\\
  \phantom{\frac{1}{2}}
D_{{-}\varpi}-\hbar\frac{q}{1-q}
  \phantom{\frac{1}{2}}
\end{subarray}\!\!\bigg],$$
and
\begin{align*}
\tfrac{1}{2}
\operatorname{tr}\big(\Theta(\varpi)^2\big)
& =
\big(D_{\varpi}+\hbar\tfrac{q}{1-q}\big)^2
+\tfrac{\hbar^2}{(1-q)(1-q^{-1})}\\
& = D_{\varpi}*D_{\varpi} + 2\hbar \tfrac{q}{1-q}D_{\varpi}
-\hbar^2 \tfrac{q}{1-q}.
\end{align*}
Theorem \ref{qHpre} gives
\begin{align}\label{eq:preforsl2}
D_{\varpi}*D_{\varpi} + 2\hbar \frac{q}{1-q}D_{\varpi}
-\hbar^2 \frac{q}{1-q}
=\varpi^2.
\end{align}
Thus $QH_{\bbT}^*(T^*\calB)$ is generated by $D_{\varpi}$ with relation \eqref{eq:preforsl2}. 

On the other hand, $T^*\calB=T^*\mathbb{P}^1$ is also a hypertoric variety. Following \cite{MR3095147}, we denote by $x$ and $y$ the fundamental class of the cotangent fibers at $0$ and $\infty$, respectively. We denote $h=-\hbar$. Then
$$x =-\varpi-D_{\varpi}\quad \text{and}\quad y =\varpi-D_{\varpi}.$$
We thus get the following relation, compatible with \cite[Theorem 1.1]{MR3095147}, 
$$x*y = q(h-x)*(h-y).$$
\end{Eg}

\begin{Rmk}\label{rmk:TequivtoG}
Note that the $\bbG$-equivariant (quantum) cohomology can be viewed as the invariants ring under the left Weyl group action recalled at the beginning of Section \ref{sec:rel} below. 
By Lemma \ref{lem:leftW}, it is not hard to obtain that 
\[QH_{\mathbb{G}}^*(T^*\calB)\cong \calO(T^\vee_{\reg})[\hbar][D_{\mu}]_{\mu\in X^*(T)}\cong \calO(T^\vee_{\reg}\times \ft^*)[\hbar].\] 
This recovers the first part of Proposition \ref{thm:BMO} (\ref{thm:BMO(3)}) by Braverman, Maulik, and Okounkov. 
Using the second part of Proposition \ref{thm:BMO} (\ref{thm:BMO(3)}), our result provides an application to the Calogero--Moser system, see Corollary \ref{thm:classicalcm}. 
\end{Rmk}

\subsection{Construction of Relations}\label{sec:rel} In this subsection, we will show that the equalities $\mathcal{R}_\lambda^k=0$ hold in $QH_{\bbT}^*(T^*\calB)$.

Recall that the left Weyl group action on $H^*_{\mathbb{T}}(T^*\calB)$ is induced by the left multiplication action of $G$ on $T^*\calB$. 
By extending it linearly over $\calO(T^\vee_{\reg})$, we get a left action of $W$ on $QH^*_{\mathbb{T}}(T^*\calB)$. For $w\in W$, we use $w^L$ to denote this action. 
By noting that $w^L$ preserves line bundles and using 
\cite[Equation (12)]{MR4465997} \footnote{The formula in \emph{loc. cit.} was stated for $QH^*_T(G/P)$, while the whole set-up works for all $G$-varieties.}, we have 
\begin{Lemma}\label{lem:leftW}
  The following holds for any 
 $f\in H^*_{\bbT}(\pt)$, $g\in \calO(T^\vee_{\reg})$, $\mu\in X^*(T)$ and $\gamma_1, \gamma_2\in QH^*_{\bbT}(T^*\calB)$:
\begin{enumerate}[\quad\rm(1)]
    \item \label{lem:leftW(1)}
        $w^L (f gD_\mu)=w(f)gD_\mu$;  
    \item $w^L (\gamma_1*\gamma_2)=(w^L\gamma_1)*(w^L\gamma_2)$.
\end{enumerate}
\end{Lemma}
\noindent 
Thus every $w^L$ is a ring automorphism of $QH^*_{\mathbb{T}}(T^*\calB)$. Moreover, the above action of $W$ commutes with the action by the affine Hecke algebra $\calH_\hbar$.

\begin{Lemma}\label{lem:Theta}
  Let $u,v,w\in W$ and $\alpha\in R^+\setminus R^+_\lambda$. We have
\begin{enumerate}[\quad\rm(1)]
  \item \label{lem:Theta(1)}
  $\Theta_{1,\o{s_\alpha}}(\lambda)=\frac{-\hbar\left<\lambda,\alpha^\vee\right>}{1-q^{\alpha^\vee}}$;
  \item \label{lem:Theta(2)}
  $T_u(\Theta_{{1},{w}})=\Theta_{\o{u},\o{uw}}$;
  \item \label{lem:Theta(3)}
  $v^L\circ T_u(\Theta_{{1},{w}})=\Theta_{\o{u},\o{uw}}$.
\end{enumerate}
\end{Lemma}
\begin{proof}
Statement (\ref{lem:Theta(1)}) follows directly from the definition of $\Theta(\lambda)$.
Next, we need to verify Statement (\ref{lem:Theta(2)}) in each case of the definition of $\Theta(\lambda)$ \eqref{eq:defofThetalambda}. 
In the first case, we have $T_u(\Theta_{1, 1})=T_u(\Delta_\lambda)=\Delta_{u\lambda}=\Theta_{u, \o{u \cdot 1}}$ by Lemma \ref{lem:TDel} (\ref{lem:TDel(2)}); the second case follows from the definition of $T_u$; if the third case occurs for $\Theta_{1, w}$ (i.e. $\Theta_{1,w}=0$) then so is $\Theta_{u, \o{uw}}$.
Lastly, since every entry in $\Theta(\lambda)$ is fixed by the left Weyl group action $v^L$ by Lemma \ref{lem:leftW} (\ref{lem:leftW(1)}), Statement (\ref{lem:Theta(3)}) follows.
\end{proof}

Let
\begin{equation}
\varsigma_\lambda=\sum_{w\in W_\lambda}(-1)^{\ell(w)}\Stab_-(w),
\end{equation}
denote the \emph{averaging class}. 
Define
\begin{equation}
\Sigma(\lambda) = \big(v^L\circ T_u (\varsigma_\lambda)\big)_{{u},{v}\in W^\lambda}.
\end{equation}
Note that $T_u (\varsigma_\lambda)$ is not fixed\footnote{Using \cite[Theorem 4.3]{MR4465997} and \cite[Proposition 9.7]{aluffi2017shadows}, we can show that $s_i^L(\varsigma_\lambda)=\frac{\hbar+\alpha_i}{\hbar-\alpha_i}\varsigma_\lambda$ for any $s_i\in W_\lambda$. }
by $W_\lambda$.

\begin{Lemma}\label{lem:Sigma} 
For any $w\in W_\lambda$ and $u, v\in W^\lambda$, we have 
\begin{enumerate}[\quad\rm(1)]
\item \label{lem:Sigma(1)}
$T_w(\varsigma_\lambda)=\varsigma_\lambda$; 
\item \label{lem:Sigma(2)} 
$\Sigma(\lambda)_{{u},{v}} =
v^L\left(\sum_{w\in uW_\lambda} (-1)^{\ell(w)}\Stab_-(w^{-1})\right)$.
\end{enumerate}
\end{Lemma}
\begin{proof}
Both statements follow immediately from Lemma \ref{lem:TDel} and the fact that $\ell(wu^{-1})\equiv \ell(w)+\ell(u) \bmod 2$ for any $u, w\in W$.
\end{proof}
 Define the diagonal matrix
\begin{equation}
\Xi(\lambda)=\operatorname{diag}\big(u\lambda\big)_{{u}\in W^\lambda}.
\end{equation}
As an application of Theorem \ref{Swisring}, we have the following key proposition. 

\begin{Prop}\label{MB=BN}
For any $\lambda\in X^*(T)_\pm$, we have
\begin{equation}
\Theta(\lambda)*\Sigma(\lambda)=\Sigma(\lambda)*\Xi(\lambda),
\end{equation}
where the product of matrices is evaluated by the quantum product.
\end{Prop}
\begin{proof}
Let us first compare the $(1,1)$-entry of both sides. By definition, 
\begin{align*}
&\quad (\Theta(\lambda)*\Sigma(\lambda))_{1,1}\\
& =\sum_{w\in W^\lambda}
\Theta(\lambda)_{1,w}*\Sigma(\lambda)_{w,1}\\
&=\Theta(\lambda)_{1,1}*\Sigma(\lambda)_{1,1} + \sum_{\alpha\in R^+\setminus R^+_\lambda}\Theta(\lambda)_{1,\o{s_\alpha}}*\Sigma(\lambda)_{\o{s_\alpha},1}\\
&=\sum_{w\in W_\lambda}(-1)^{\ell(w)}\Delta_\lambda*\Stab_-(w)
  +\hbar\!\!\!\!\sum_{\alpha\in R^+\setminus R^+_\lambda}\frac{\langle\lambda,\alpha^\vee\rangle}{1-q^{\alpha^\vee}}\sum_{w\in W_\lambda}(-1)^{\ell(w)}\Stab_-(ws_\alpha)\\
&=\lambda\cdot\varsigma_\lambda
=\lambda \cdot\Sigma(\lambda)_{1,1}.
\end{align*}
Here the third equality uses the equality $T_{\o{s_\alpha}}(\varsigma_\lambda)=T_{{s_\alpha}}(\varsigma_\lambda)$ by Lemma \ref{lem:Sigma} (\ref{lem:Sigma(1)}) as well as the fact $(-1)^{\ell(s_\alpha)}=-1$.
The forth equality follows from \eqref{eq:qChe}, and the facts that (i) $w(\lambda)=\lambda$ for any $w\in W_\lambda$, (ii) $\langle\lambda,\alpha^\vee\rangle=0$ for every $\alpha\in R^+_\lambda$, and (iii) $w\alpha>0$ for every $w\in W_\lambda$ and $\alpha\in R^+\setminus R^+_\lambda$.

For general $u,v\in W^\lambda$, by applying $v^L\circ T_u$ to the above equalities, we get
\begin{align*}
v(\lambda)\cdot \Sigma(\lambda)_{u,v}=v^L\circ T_u(\lambda\cdot \Sigma(\lambda)_{1,1})
&=v^L\circ T_u\bigg(\sum_{w\in W^\lambda}
\Theta(\lambda)_{1,w}*\Sigma(\lambda)_{w,1}\bigg)\\
& = \sum_{w\in W^\lambda}\Theta(\lambda)_{u,\o{uw}}*\Sigma(\lambda)_{\o{uw},v}\\
& = \sum_{w\in W^\lambda}
\Theta_{u,w}(\lambda)*\Sigma(\lambda)_{w,v}\\
& =(\Theta(\lambda)*\Sigma(\lambda))_{u,v}.
\end{align*}
Here the second equality follows from Lemmas \ref{lem:leftW}, \ref{lem:Theta}, \ref{lem:Sigma} and Theorem \ref{Swisring}.
This finishes the proof of the proposition.
\end{proof}

\begin{Eg}
Let $G=SL(2, \bbC)$, and use the same notation as in Example \ref{eg:autoP1}. 
Then Proposition \ref{MB=BN} gives
$$
\bigg[\,\begin{subarray}{c}
  \phantom{\frac{1}{2}}
\Delta_{\varpi}
  \phantom{\frac{1}{2}}\\
\frac{-\hbar}{1-q^{{-}1}}
\end{subarray}
\begin{subarray}{c}
\frac{-\hbar}{1-q}\\
  \phantom{\frac{1}{2}}
\Delta_{-\varpi}
  \phantom{\frac{1}{2}}
\end{subarray}\!\!\bigg]
*
\bigg[\begin{subarray}{c}
\Stab_-(\id)\\[0.5ex]
-\Stab_-(s)\\
\end{subarray}\,\,
\begin{subarray}{c}
s^L\Stab_-(\id)\\[0.5ex]
-s^L\Stab_-(s)
\end{subarray}\bigg]
=\bigg[\begin{subarray}{c}
\Stab_-(\id)\\[0.5ex]
-\Stab_-(s)\\
\end{subarray}\,\,
\begin{subarray}{c}
s^L\Stab_-(\id)\\[0.5ex]
-s^L\Stab_-(s)
\end{subarray}\bigg]
*
\bigg[\begin{subarray}{c}
\varpi\\[1ex]
\phantom{0}\\
\end{subarray}\,\,\,\,
\begin{subarray}{c}
\phantom{0}\\[1ex]
-\varpi
\end{subarray}\bigg]
$$
\end{Eg}

 The next lemma works for any commutative algebra  over an arbitrary field $\mathbb{F}$.

\begin{Lemma}\label{lem: relation} Let $\Theta,\Sigma$ be square matrices of order $m$ with entries in an $\mathbb{F}$-algebra $\mathcal{A}$, and  $\Xi=\operatorname{diag}(a_1,\ldots, a_m)$ with $a_1,\ldots,a_m$ distinct elements in $\mathbb{F}$.
Assume 
$$\Theta \cdot \Sigma=\Sigma \cdot \Xi. $$
If there is at least one unit in each column of $\Sigma$, then $\det(\Sigma)$ is a unit in $\mathcal{A}$.
\end{Lemma}
\begin{proof}
Assume $\det(\Sigma)$ is not a unit in $\mathcal{A}$.
Then there is a maximal ideal $\mathfrak{m}$ containing $\det(\Sigma)$, so that $\overline {\mathcal{A}}:=\mathcal{A}/\mathfrak{m}$ is a field and 
$\overline{\det(\Sigma)}=0\in \overline{\mathcal{A}}$. Note that the columns of $\overline{\Sigma}$ are all nonzero vectors by our assumption, and hence are eigenvectors of $\overline{\Theta}$ of distinct eigenvalues by the hypothesis.
Therefore they form a basis of $\overline{\mathcal{A}}^m$. In other words, $ \overline{\det(\Sigma)}=\det(\overline{\Sigma})$ is invertible, resulting in a contradiction.
\end{proof}

Let us denote $\mathbb{F}=\mbox{Frac} H^*_{\bbT}(\pt)(\!(\hbar^{-1})\!)$, and consider the $\mathbb{F}$-algebra 
\begin{equation}\label{eq:defofcalA}
  \mathcal{A}=QH_{\bbT}^*(T^*\calB)\otimes_{H_{\bbT}^*(\pt)}\mathbb{F}.
\end{equation}

\begin{Prop}\label{detLemma} 
For any $\lambda\in X^*(T)_\pm$, the determinant $\det\Sigma(\lambda)$ is a unit in $\mathcal{A}$.
\end{Prop}
\begin{proof} For any $v\in W^\lambda$, by \cite[Proposition 9.7 and Lemma 7.3]{aluffi2017shadows}, we have
$$\Sigma(\lambda)_{1,{v}} = (-1)^{\dim \calB}\hbar^{\dim \calB}+\big(\text{terms of lower degrees in $\hbar$}\big),$$
Hence, it is a unit in $\mathcal{A}$.
Thus every column of $\Sigma(\lambda)$ contains a unit in $\mathcal{A}$ (at the first row). 
Moreover, the diagonal terms of $\Xi(\lambda)$ are distinct elements in $\bbF$. 
By Lemma \ref{lem: relation}, $\det\Sigma(\lambda)$ is a unit in $\mathcal{A}$. 
\end{proof}

\begin{Prop}
 \label{Mainprop}
For any $\lambda\in X^*(T)_\pm$ and $k\in \mathbb{Z}_{>0}$, in $QH^*_{\bbT}(T^*\calB)$, we have
\begin{equation} \label{mainrelation}
\operatorname{tr}\big(\Theta(\lambda)^k\big)=\sum_{{u}\in W^\lambda} (u\lambda)^k.
\end{equation}
\end{Prop}

\begin{proof}
By Proposition \ref{detLemma}, $\det\Sigma(\lambda)$ is a unit in $\mathcal{A}$, so that $\Sigma(\lambda)$ is invertible.
Therefore by Proposition \ref{MB=BN}, we have the following relation in $\mathcal{A}$:
$$\operatorname{tr} \big(\Theta(\lambda)^k\big)=\operatorname{tr} \big(\Sigma(\lambda)^{-1}* \Theta(\lambda)^k *\Sigma(\lambda)\big)=\operatorname{tr} \big(\Xi(\lambda)^k\big)=\sum\nolimits_{u\in W^\lambda} (u\lambda)^k.$$ 
 Notice that there is a natural embedding 
 $QH^*_{\mathbb{T}}(T^*\calB)\hookrightarrow QH^*_{\mathbb{T}}(T^*\calB)\otimes_{H_{\bbT}(\pt)} \mathbb{F}=\mathcal{A}$, and both sides of the
 equality \eqref{mainrelation} take values in $QH^*_{\bbT}(T^*\calB)$,
 the statement follows.
\end{proof}

\subsection{Proof of Theorem \ref{qHpre}} By Proposition \ref{Mainprop}, all $\mathcal{R}_\lambda^k$ are in the kernel of $\pi$. 
Once we prove that they generate the entire kernel, this will finish our proof.
The following Deformation Principle is a key ingredient in our argument, the proof of which is similar to that of a lemma by Siebert and Tian \cite{MR1621570}.

\begin{Prop}[Deformation Principle]\label{DefPrin}
Let $\{\mathcal{R}_i\}_{i\in I}$ be a family of elements in $\ker \pi$. Assume for all $i\in I$
$$\mathcal{R}_i = f_i(D_{\mu})-f_i(\mu)\mod \hbar,$$
where $\{f_i\}_{i\in I}$ generates $\Sym(\ft^*)^W$. 
Then $\ker \pi$ is generated by $\{\mathcal{R}_i\}_{i\in I}$. 
\end{Prop}
\begin{proof}
We have the following diagram
$$\xymatrix{
\calO(T_{\reg}^\vee)[\hbar][\mu,D_{\mu}]_{\mu}\ar[r]^{\quad\pi}\ar[d]& QH_{\bbT}^*(T^*\calB)\ar[d]\\
\calO(T_{\reg}^\vee)[\mu,D_{\mu}]_{\mu}\ar[r]_{\varphi}& QH_{T}^*(T^*\calB),
}$$
where the vertical maps are modulo $\hbar$.
 Since $T^*\calB$ is a homomorphic symplectic variety, the quantum product and the equivariant product coincide modulo $\hbar$ \cite{MR2782198}. 
Hence, by Proposition \ref{Th:BorelTh} and isomorphism \eqref{eq:HTGBh}, we see that
$$\ker\varphi =\big<f(D_{\mu})-f(\mu)\mid f(\mu)\in \Sym(\ft^*)^W\big>.$$
 
Let $J$ be the ideal generated by $\{\mathcal{R}_i\}_{i\in I}$. 
Assume $J\neq \ker \pi$. Let $\gamma$ be an element of minimal $\hbar$-degree such that $\gamma\in \ker \pi\setminus J$. By the commutativity of the above diagram,
\[\varphi(\gamma\bmod \hbar)=\pi(\gamma)\bmod \hbar =0.\]
Note that $J/\left<\hbar\right>=\ker \varphi$ by our assumption.  
We can write
$$\gamma=\gamma_1+\hbar\gamma_2$$
for $\gamma_1\in J$ and some $\gamma_2\notin J$ of $\hbar$-degree less than $\gamma$. Since $J\subset \ker \pi$, we have $\hbar\pi(\gamma_2)=0$, which implies $\pi(\gamma_2)=0$. Hence, $\gamma_2\in \ker\pi\setminus J$, and the $\hbar$-degree of $\gamma_2$ is less than that of $\gamma$, which contradicts the choice of $\gamma$.
\end{proof}

\begin{proof}[Proof of Theorem \ref{qHpre}]
First of all, we have $\mathcal{R}_{\lambda}^k\in \ker \pi$ by Proposition \ref{Mainprop}. Modulo $\hbar$, the matrix $\Theta(\lambda)$ becomes a diagonal matrix with entries $D_{u\lambda}$ for ${u}\in W^{\lambda}$. Hence,
$$\mathcal{R}_{\lambda}^k=
\sum_{ {u}\in W^\lambda} (D_{u\lambda})^k
-\sum_{ {u}\in W^\lambda} (u\lambda)^k\mod \hbar.$$
By {\cite[Section 23.1]{MR499562}},
the set of polynomials
 $\sum_{u\in W^\lambda} (u\lambda)^k$
for all dominant weights $\lambda$ and $k\geq 1$ spans the space $\Sym(\ft^*)^W$.
Therefore the family of all $\mathcal{R}_{\lambda}^k$ satisfies the condition of the Deformation Principle \ref{DefPrin}. We can conclude that all $\mathcal{R}_{\lambda}^k$ generate the kernel of $\pi$.
\end{proof}

\subsection{Explicit formulae for the classical Calogero--Moser map}\label{sec:CMmap}
In this subsection, we will give an explicit formula for the classical trigonometric Calogero--Moser map \eqref{equ:etaCcl}.
Recall that $q^{\alpha^\vee}$ is naturally a function over $T^\vee_{\reg}$ for any coroot $\alpha^\vee$
and ${p}_{\lambda}$ defined in \eqref{eqn:partial} is the linear function over $\ft$.
Both of them can be viewed as functions over $T^\vee_{\reg}\times \ft$ naturally.

Let $Y(\lambda)$ be the matrix obtained from $\Theta(\lambda)$ by replacing $\Delta_{u\lambda}$ by $p_{u\lambda}$ in \eqref{eq:defofThetalambda}. 
By Proposition \ref{thm:BMO} (\ref{thm:BMO(3)}), we get immediately

\begin{Coro}\label{thm:classicalcm}
We have 
$$\eta_{\CM,\hbar}^{\cl}\bigg(\sum_{u\in W^{\lambda}}(u\lambda)^k\bigg)=\operatorname{tr}(Y(\lambda)^k)\in \calO(T^\vee_{\reg}\times \ft)[\hbar].$$
\end{Coro}

\begin{Rmk}\label{Rmk:Etingof}
P. Etingof shared with us in private communication that our matrix $Y(\lambda)$ can be viewed as the matrix of the Dunkl operator under the basis $W$ of $\calO(T^\vee_{\reg}\times \ft)[\hbar]\rtimes \mathbb{C}[W]$, provided that $\lambda$ is strictly dominant (i.e., $W_{\lambda}$ is trivial). In the rational case (analogous to our case), this method yields the known explicit formula for the rational Calogero--Moser map in both classical and quantum cases, obtained based on Lax pairs \cite{shastry1993super}. In the trigonometric case (our case), Corollary \ref{thm:classicalcm} can be translated to the following trace-free formula:
\begin{equation}\label{eq:tracefree}
\operatorname{tr}\left(
\sum_{w\in W} \Dun_{w\lambda}^k-
\sum_{w\in W} w\Dun_{\lambda}^kw^{-1}\right)=0.
\end{equation}
Recall that $\Dun_{\lambda}$ is the classical trigonometric Dunkl operator defined in \eqref{eq:cltriDunkop}.
It is worth noting that the trigonometric Dunkl operators and Weyl group elements generate the degenerate Hecke algebra $\calH_{\hbar}$, and thus, unlike the rational case, the two sums in \eqref{eq:tracefree} are not equal.
To relate our result to the above formula, we need to note that
\begin{align*}
\operatorname{tr}(Y(\lambda)^k)
& =\operatorname{tr}(\Dun_{\lambda}^k)
= \frac{1}{|W|}\operatorname{tr}\left(\sum_{w\in W} w\Dun_{\lambda}^kw^{-1}\right).
\end{align*}
On the other hand, by extending \cite[Lemma 2.2]{MR2769318} to the trigonometric case, we can conclude that $\sum_{w\in W} \Dun_{w\lambda}^k$ is the scalar matrix of $\eta_{\CM,\hbar}^{\cl}\left(\sum_{w\in W}(w\lambda)^k\right)$. Hence,
\begin{align*}
\eta_{\CM,\hbar}^{\cl}\left(\sum_{w\in W}(w\lambda)^k\right)
& = \frac{1}{|W|}\operatorname{tr}\left(\sum_{w\in W} \Dun_{w\lambda}^k\right).
\end{align*}
Corollary \ref{thm:classicalcm} is thus equivalent to \eqref{eq:tracefree}. 
It is interesting to ask for a direct proof of \eqref{eq:tracefree}. 
Additionally, it should be noted that the trace-free formula \eqref{eq:tracefree} is not valid for the quantum case.

In practical computation, rather than substituting square matrices of size $|W|$ into a symmetric polynomial, our matrix is small and involves only matrix powers. This reduces the computational complexity significantly.
\end{Rmk}

\begin{Eg}[Hamiltonian]
Let $\lambda\in X^*(T)_+$ be strictly dominant such that $W_\lambda$ is trivial.
By direct computations, we have
\begin{align*}
\operatorname{tr}\big(Y(\lambda)^2\big)
& =\sum_{w\in W} {p}_{w\lambda}^2
+\hbar^2\sum_{w\in W}
\sum_{\alpha>0}\left<\lambda,\alpha^\vee\right>^2
\frac{1}{\big(1-q^{w\alpha^\vee}\big)\big(1-q^{-w\alpha^\vee}\big)}\\
& = \sum_{w\in W} {p}_{w\lambda}^2
-\hbar^2
\sum_{\alpha>0}
\bigg(\sum_{w\in W}\left<w\lambda,\alpha^\vee\right>^2\bigg)
\frac{1}{\big(q^{\alpha^\vee/2}-q^{-\alpha^\vee/2}\big)^2}.
\end{align*}
Let $C \in \Sym^2(\ft^*)$ be a non-degenerate quadratic form, which is unique up to a scalar.
Thus,
$$0\neq \sum_{w\in W} (w\lambda)^2=c_{\lambda}\cdot C\in \Sym^2(\ft^*),$$
for some constant $c_\lambda$.
Then we can rewrite the above by
$$
c_{\lambda}\cdot\left(C({p})- \hbar^2
\sum_{\alpha>0}
\frac{C(\alpha^\vee,\alpha^\vee)}{\big(q^{\alpha^\vee/2}-q^{-\alpha^\vee/2}\big)^2}
\right).$$
Corollary \ref{thm:classicalcm} shows
$$\eta_{\CM,\hbar}^{\cl}(C)=C({p})- \hbar^2
\sum_{\alpha>0}
\frac{C(\alpha^\vee,\alpha^\vee)}{\big(q^{\alpha^\vee/2}-q^{-\alpha^\vee/2}\big)^2},$$
which is the Hamiltonian of the classical Calogero--Moser system \eqref{equ:etaCcl}.
\end{Eg}

\section{Computation for Classical Types}\label{sec:ABCD}

In this section, we compute $QH_{\bbT}^*(T^*\calB)$ for groups $G$ of classical types. 

\subsection{Type $A_{n-1}$}

Recall that $G=SL(n,\mathbb{C})$ in type $A_{n-1}$. 
While for the convenience of our computation, we prefer to use $GL$-notations. 
Let $\hat{T}$ be the subgroup of diagonal matrices of $\hat{G}:=GL(n,\mathbb{C})$ and $\hat{\ft}$ be its Lie algebra.  
We denote $\hat{\bbG}=\hat{G}\times\bbC^*$ and $\hat{\bbT}=\hat{T}\times \bbC^*$. 
By carefully examining the proof, all the argument in this paper works for $GL$ with minor modification after replacing $\ft$ by $\hat{\ft}$ in all the arguments. 

Let us denote $\mathbf{e}_i\in X^*(\hat{T})$ the projection of the $i$-th diagonal entry. Thus
$$X^*(\hat{T}) = \mathbb{Z}\mathbf{e}_1\oplus \cdots \oplus \mathbb{Z}\mathbf{e}_n.$$
Denote by $\{\mathbf{e}_j^\vee\}$ the dual basis of $\{\mathbf{e}_i\}$. The  simple coroots are given by 
$$\alpha_i^\vee=\mathbf{e}^\vee_i-\mathbf{e}^\vee_{i+1}\quad (i=1,\ldots,n-1).$$
We denote $q_i=q^{\mathbf{e}^\vee_i}$.
Note that $q_i/q_j=q^{\mathbf{e}_i^\vee-\mathbf{e}_j^\vee}\in \Eff(\calB)$ for $i<j$. 
 Following the notation from classical Schubert calculus (e.g. \cite{AFbook}), we denote
$$x_i=D_{-\mathbf{e}_i},\qquad t_i=-\mathbf{e}_i.$$
We denote also
$$\chi_i=\Delta_{-\mathbf{e}_i}
=x_i+\hbar\sum_{a<i}
\frac{q_a/q_i}{1-q_a/q_i}
-\hbar\sum_{i<b}
\frac{q_i/q_b}{1-q_i/q_b}.
$$
By direct computation with respect to the ordered set $W^{-\mathbf{e}_1}=\{\id, s_1, \ldots, s_{n-1}\cdots s_1\}$, we have
\begin{equation}\label{eq:MatA}
M(\chi):=\Theta(-\mathbf{e}_1)=\left[
\begin{matrix}
\chi_1 &
\frac{\hbar}{1-q_1/q_2}&
\frac{\hbar}{1-q_1/q_3}&
\cdots &
\frac{\hbar}{1-q_1/q_n}\\[1ex]
\frac{\hbar}{1-q_2/q_1} &
\chi_2 &
\frac{\hbar}{1-q_2/q_3} &
\cdots &
\frac{\hbar}{1-q_2/q_n}\\[1ex]
\frac{\hbar}{1-q_3/q_1} &
\frac{\hbar}{1-q_3/q_2} &
\chi_3 &
\cdots &
\frac{\hbar}{1-q_3/q_n}\\[1ex]
\vdots&\vdots&\vdots&\ddots&\vdots\\[1ex]
\frac{\hbar}{1-q_n/q_1} &
\frac{\hbar}{1-q_n/q_2} &
\frac{\hbar}{1-q_n/q_3} &
\cdots &
\chi_n
\end{matrix}\right].
\end{equation}

From the proof of Proposition \ref{Mainprop}, we have the following identity over $\mathcal{A}$
$$\Sigma(-\mathbf{e}_1)^{-1}*M(\chi)*\Sigma(-\mathbf{e}_1)=\Xi(-\mathbf{e}_1)=\operatorname{diag}(t_1,\ldots,t_n).$$
Let $y$ be an indeterminant. We thus have
$$\det\bigg(y 1_n+M(\chi)\bigg)=\prod_{i=1}^n (y+t_i).$$
Expand the characteristic polynomial 
$$\det\bigg(y 1_n+M(\chi)\bigg)=
y^n+\mathcal{E}_1(\chi)y^{n-1}+\cdots+\mathcal{E}_{n-1}(\chi)y+\mathcal{E}_n(\chi).$$
Thus the following relation holds in $QH_{\hat{\bbT}}^*(T^*\calB)$,
$$\mathcal{E}_k(\chi)=e_k(t)$$
where $e_k$ is the $k$-th elementary symmetric polynomial.
Since the algebra of $W$-invariant polynomials is generated by the $k$-th elementary symmetric polynomials.
By Proposition \ref{DefPrin}, we can conclude the following presentation of quantum cohomology.

\begin{Th}\label{Th:typeApre}
For type $A_{n-1}$, we have
\begin{equation}
QH^*_{\hat{\mathbb{T}}}(T^*\calB) =
\frac{
\calO(T^\vee_{\reg})[\hbar,x_1,\ldots,x_n,t_1,\ldots,t_n]
}{\left<\mathcal{E}_k(\chi)-e_k(t),k=1,\ldots,n\right>}.
\end{equation}
Consequently, since $T$ is cut from $\hat{T}$ by the relation $\mathbf{e}_1+\cdots+\mathbf{e}_n=0$, we can conclude that 
$$QH^*_{{\mathbb{T}}}(T^*\calB)=QH^*_{\hat{\mathbb{T}}}(T^*\calB)/\left<e_1(t)\right>.$$
\end{Th}

Next, we will give a combinatorial formula for $\mathcal{E}_k(\chi)$. 
Let $K$ be a subset of $\{1,\ldots,n\}$ of order $k$. A \emph{matching} of $K$ is a map $\pi: K\to K$ with $\pi^2=\operatorname{id}$, i.e. an involution over $K$. A matching $\pi$ is called perfect, if the involution $\pi$ has no fixed points.
Let us denote
$$
J(\pi)=\prod_{\pi(i)=i}\chi_i\quad\text{and}\quad
V(\pi) = \prod_{\pi(i)=j>i}\frac{\hbar^2q_{i}q_{j}}{(q_{i}-q_{j})^2}.
$$
Note that $\frac{q_{i}q_{j}}{(q_{i}-q_{j})^2}
=\frac{q_{i}/q_j}{(1-q_{i}/q_{j})^2}\in \mathcal{O}(T^\vee_{\reg})$.

\begin{Th}\label{combforEk}
For any $1\leq k\leq n$, 
\begin{equation}
\mathcal{E}_k(\chi)=\sum_{K}\sum_{\pi} J(\pi)V(\pi),
\end{equation}
where the sum over all $k$-subsets $K$ of $\{1,\ldots,n\}$, and matchings $\pi$ of $K$.
\end{Th}

The proof of the above theorem involves some interesting equalities from the viewpoint of combinatorics. We leave the details in Appendix \ref{sec:Acomb}.

\begin{Eg}For $n=3$, 
\begin{align*}
\mathcal{E}_1(\chi) & = \chi_1+\chi_2+\chi_3\\
\mathcal{E}_2(\chi) & = \chi_1\chi_2+\chi_1\chi_3+\chi_2\chi_3
+\frac{\hbar^2q_1q_2}{(q_1-q_2)^2}
+\frac{\hbar^2q_1q_3}{(q_1-q_3)^2}
+\frac{\hbar^2q_2q_3}{(q_2-q_3)^2}\\
\mathcal{E}_3(\chi) & = \chi_1\chi_2\chi_3+
+\chi_3\frac{\hbar^2q_1q_2}{(q_1-q_2)^2}
+\chi_2\frac{\hbar^2q_1q_3}{(q_1-q_3)^2}
+\chi_1\frac{\hbar^2q_2q_3}{(q_2-q_3)^2}.
\end{align*}
We can illustrate each term above by drawing the matching as follows
$$\setlength{\unitlength}{1pc}
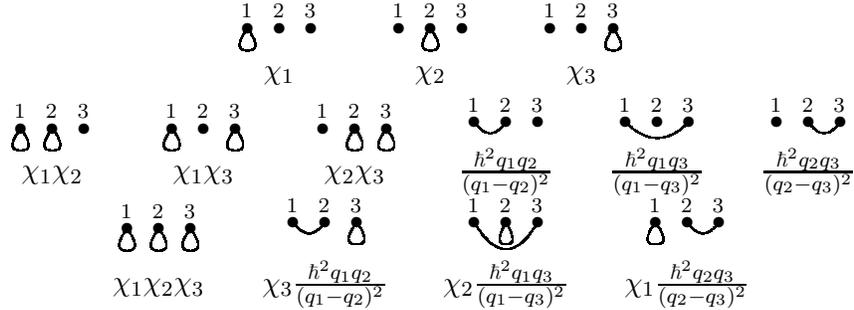
\begin{array}{c}
\begin{picture}(3,1)
\put(0.3,0.5){${}^1$}
\put(0.5,0.5){\circle*{0.3}}
\put(1.3,0.5){${}^2$}
\put(1.5,0.5){\circle*{0.3}}
\put(2.3,0.5){${}^3$}
\put(2.5,0.5){\circle*{0.3}}
\qbezier(0.5,0.5)(0,-0.2)(0.5,-0.2)
\qbezier(0.5,0.5)(1,-0.2)(0.5,-0.2)
\end{picture}\\
\chi_1
\end{array}\quad
\begin{array}{c}
\begin{picture}(3,1)
\put(0.3,0.5){${}^1$}
\put(0.5,0.5){\circle*{0.3}}
\put(1.3,0.5){${}^2$}
\put(1.5,0.5){\circle*{0.3}}
\put(2.3,0.5){${}^3$}
\put(2.5,0.5){\circle*{0.3}}
\qbezier(1.5,0.5)(1,-0.2)(1.5,-0.2)
\qbezier(1.5,0.5)(2,-0.2)(1.5,-0.2)
\end{picture}\\
\chi_2
\end{array}\quad
\begin{array}{c}
\begin{picture}(3,1)
\put(0.3,0.5){${}^1$}
\put(0.5,0.5){\circle*{0.3}}
\put(1.3,0.5){${}^2$}
\put(1.5,0.5){\circle*{0.3}}
\put(2.3,0.5){${}^3$}
\put(2.5,0.5){\circle*{0.3}}
\qbezier(2.5,0.5)(2,-0.2)(2.5,-0.2)
\qbezier(2.5,0.5)(3,-0.2)(2.5,-0.2)
\end{picture}\\
\chi_3
\end{array}$$
$$\setlength{\unitlength}{1pc}
\begin{array}{c}
\begin{picture}(3,1)
\put(0.3,0.5){${}^1$}
\put(0.5,0.5){\circle*{0.3}}
\put(1.3,0.5){${}^2$}
\put(1.5,0.5){\circle*{0.3}}
\put(2.3,0.5){${}^3$}
\put(2.5,0.5){\circle*{0.3}}
\qbezier(0.5,0.5)(0,-0.2)(0.5,-0.2)
\qbezier(0.5,0.5)(1,-0.2)(0.5,-0.2)
\qbezier(1.5,0.5)(1,-0.2)(1.5,-0.2)
\qbezier(1.5,0.5)(2,-0.2)(1.5,-0.2)
\end{picture}\\
\chi_1\chi_2
\end{array}\quad
\begin{array}{c}
\begin{picture}(3,1)
\put(0.3,0.5){${}^1$}
\put(0.5,0.5){\circle*{0.3}}
\put(1.3,0.5){${}^2$}
\put(1.5,0.5){\circle*{0.3}}
\put(2.3,0.5){${}^3$}
\put(2.5,0.5){\circle*{0.3}}
\qbezier(0.5,0.5)(0,-0.2)(0.5,-0.2)
\qbezier(0.5,0.5)(1,-0.2)(0.5,-0.2)
\qbezier(2.5,0.5)(2,-0.2)(2.5,-0.2)
\qbezier(2.5,0.5)(3,-0.2)(2.5,-0.2)
\end{picture}\\
\chi_1\chi_3
\end{array}\quad
\begin{array}{c}
\begin{picture}(3,1)
\put(0.3,0.5){${}^1$}
\put(0.5,0.5){\circle*{0.3}}
\put(1.3,0.5){${}^2$}
\put(1.5,0.5){\circle*{0.3}}
\put(2.3,0.5){${}^3$}
\put(2.5,0.5){\circle*{0.3}}
\qbezier(1.5,0.5)(1,-0.2)(1.5,-0.2)
\qbezier(1.5,0.5)(2,-0.2)(1.5,-0.2)
\qbezier(2.5,0.5)(2,-0.2)(2.5,-0.2)
\qbezier(2.5,0.5)(3,-0.2)(2.5,-0.2)
\end{picture}\\
\chi_2\chi_3
\end{array}\quad
\begin{array}{c}
\begin{picture}(3,1)
\put(0.3,0.5){${}^1$}
\put(0.5,0.5){\circle*{0.3}}
\put(1.3,0.5){${}^2$}
\put(1.5,0.5){\circle*{0.3}}
\put(2.3,0.5){${}^3$}
\put(2.5,0.5){\circle*{0.3}}
\qbezier(0.5,0.5)(1,-0.2)(1.5,0.5)
\end{picture}\\
\frac{\hbar^2q_1q_2}{(q_1-q_2)^2}
\end{array}\quad
\begin{array}{c}
\begin{picture}(3,1)
\put(0.3,0.5){${}^1$}
\put(0.5,0.5){\circle*{0.3}}
\put(1.3,0.5){${}^2$}
\put(1.5,0.5){\circle*{0.3}}
\put(2.3,0.5){${}^3$}
\put(2.5,0.5){\circle*{0.3}}
\qbezier(0.5,0.5)(1.5,-0.5)(2.5,0.5)
\end{picture}\\
\frac{\hbar^2q_1q_3}{(q_1-q_3)^2}
\end{array}\quad
\begin{array}{c}
\begin{picture}(3,1)
\put(0.3,0.5){${}^1$}
\put(0.5,0.5){\circle*{0.3}}
\put(1.3,0.5){${}^2$}
\put(1.5,0.5){\circle*{0.3}}
\put(2.3,0.5){${}^3$}
\put(2.5,0.5){\circle*{0.3}}
\qbezier(1.5,0.5)(2,-0.2)(2.5,0.5)
\end{picture}\\
\frac{\hbar^2q_2q_3}{(q_2-q_3)^2}
\end{array}$$
$$\setlength{\unitlength}{1pc}
\begin{array}{c}
\begin{picture}(3,1)
\put(0.3,0.5){${}^1$}
\put(0.5,0.5){\circle*{0.3}}
\put(1.3,0.5){${}^2$}
\put(1.5,0.5){\circle*{0.3}}
\put(2.3,0.5){${}^3$}
\put(2.5,0.5){\circle*{0.3}}
\qbezier(0.5,0.5)(0,-0.2)(0.5,-0.2)
\qbezier(0.5,0.5)(1,-0.2)(0.5,-0.2)
\qbezier(1.5,0.5)(1,-0.2)(1.5,-0.2)
\qbezier(1.5,0.5)(2,-0.2)(1.5,-0.2)
\qbezier(2.5,0.5)(2,-0.2)(2.5,-0.2)
\qbezier(2.5,0.5)(3,-0.2)(2.5,-0.2)
\end{picture}\\[1ex]
\chi_1\chi_2\chi_3
\end{array}\quad
\begin{array}{c}
\begin{picture}(3,1)
\put(0.3,0.5){${}^1$}
\put(0.5,0.5){\circle*{0.3}}
\put(1.3,0.5){${}^2$}
\put(1.5,0.5){\circle*{0.3}}
\put(2.3,0.5){${}^3$}
\put(2.5,0.5){\circle*{0.3}}
\qbezier(0.5,0.5)(1,-0.2)(1.5,0.5)
\qbezier(2.5,0.5)(2,-0.2)(2.5,-0.2)
\qbezier(2.5,0.5)(3,-0.2)(2.5,-0.2)
\end{picture}\\[1ex]
\chi_3\frac{\hbar^2q_1q_2}{(q_1-q_2)^2}
\end{array}\quad
\begin{array}{c}
\begin{picture}(3,1)
\put(0.3,0.5){${}^1$}
\put(0.5,0.5){\circle*{0.3}}
\put(1.3,0.5){${}^2$}
\put(1.5,0.5){\circle*{0.3}}
\put(2.3,0.5){${}^3$}
\put(2.5,0.5){\circle*{0.3}}
\qbezier(0.5,0.5)(1.5,-1.2)(2.5,0.5)
\qbezier(1.5,0.5)(1.1,-0.2)(1.5,-0.2)
\qbezier(1.5,0.5)(1.9,-0.2)(1.5,-0.2)
\end{picture}\\[1ex]
\chi_2\frac{\hbar^2q_1q_3}{(q_1-q_3)^2}
\end{array}\quad
\begin{array}{c}
\begin{picture}(3,1)
\put(0.3,0.5){${}^1$}
\put(0.5,0.5){\circle*{0.3}}
\put(1.3,0.5){${}^2$}
\put(1.5,0.5){\circle*{0.3}}
\put(2.3,0.5){${}^3$}
\put(2.5,0.5){\circle*{0.3}}
\qbezier(1.5,0.5)(2,-0.2)(2.5,0.5)
\qbezier(0.5,0.5)(0,-0.2)(0.5,-0.2)
\qbezier(0.5,0.5)(1,-0.2)(0.5,-0.2)
\end{picture}\\[1ex]
\chi_1\frac{\hbar^2q_2q_3}{(q_2-q_3)^2}
\end{array}$$
\end{Eg}

\subsection{Type $B_n$ and type $C_n$}
The weight lattice is naturally recognized as
$$\mathbb{Z}\mathbf{e}_1\oplus \cdots \oplus \mathbb{Z}\mathbf{e}_n$$
with the set of simple coroots
$$\alpha^\vee_i=\mathbf{e}^\vee_i-\mathbf{e}^\vee_{i+1}\quad (i=1,\ldots,n-1)\qquad
\alpha_n^\vee=p\mathbf{e}^\vee_n$$
where $p=2$ for type $B_n$ and $p=1$ for type $C_n$.
We denote
$$q_i=q^{\mathbf{e}^\vee_i},\qquad x_i=D_{-\mathbf{e}_i},\qquad\chi_i=\Delta_{-\mathbf{e}_i},\qquad t_i=-\mathbf{e}_i.$$
With respect to the ordered set 
 $$W^{-\mathbf{e}_1}=\big\{\id, s_1, {\ldots},  s_n{\cdot}{\cdot}{\cdot} s_1, s_{n\texttt{-}1}s_n{\cdot}{\cdot}{\cdot} s_1,{\ldots},  s_1{\cdot}{\cdot}{\cdot} s_{n\texttt{-}1} s_n{\cdot}{\cdot}{\cdot} s_1\big\},$$
we have
\begin{equation}\label{eq:MatBC}
M(\chi):=\Theta(-\mathbf{e}_1)=\left[
\begin{array}{@{}c@{\,}c@{\,}c|c@{\,}c@{\,}c@{}}
\chi_1 & \cdots&\frac{\hbar}{1-q_1/q_n}&
  \frac{\hbar}{1-q_1q_n} & \cdots & \frac{p\hbar}{1-q_1^p}\\[1ex]
\vdots & \ddots&\vdots&
  \vdots&\reflectbox{$\ddots$}&\vdots\\[1ex]
\frac{\hbar}{1-q_n/q_1} & \cdots & \chi_n&
  \frac{p\hbar}{1-q_n^p} & \cdots & \frac{\hbar}{1-q_1q_n}\\
    [1ex]\hline\rule{0pc}{3ex}
\frac{\hbar}{1-q_n^{-1}q_1^{-1}} & \cdots &\frac{p\hbar}{1-q_n^{-p}} &
  -\chi_n&\cdots & \frac{\hbar}{1-q_1/q_n}\\[1ex]
\vdots&\reflectbox{$\ddots$}&\vdots&
  \vdots & \ddots&\vdots\\[1ex]
\frac{p\hbar}{1-q_1^{-p}} & \cdots & \frac{\hbar}{1-q_1^{-1}q_n^{-1}}&
  \frac{\hbar}{1-q_n/q_1} & \cdots & -\chi_1
\end{array}\right].
\end{equation}
Here in the upper right (resp. lower left) block, the $(i,j)$-entry is given by $\frac{\hbar}{1-q_iq_{n+1-j}}$ (resp. $\frac{\hbar}{1-q_{n+1-i}^{-1}q_{j}^{-1}}$) when $i+j\neq n+1$ and ${p\hbar\over 1-q_i^p}$ (resp. ${p\hbar\over 1-q_{j}^{-p}}$) otherwise. 
 Similar to the case of type $A$, we have
$$\det\bigg(y 1_{2n}+M(\chi)\bigg)=\prod_{i=1}^{n} (y^2-t_i^2).$$
Expand the characteristic polynomial 
$$\det\bigg(y 1_{2n}+M(\chi)\bigg)=
y^{2n}+\mathcal{E}_1(\chi)y^{2n-1}+\cdots+\mathcal{E}_{2n-1}(\chi)y+\mathcal{E}_{2n}(\chi).$$
Thus for each $k\in \{1,\ldots, n\}$, $\mathcal{E}_{2k-1}(\chi)=0$, and the following relation holds in $QH_{\bbT}^*(T^*\calB)$,
$$\mathcal{E}_{2k}(\chi)=(-1)^ke_k(t^2).$$
Since the algebra of $W$-invariant polynomial is generated by $k$-th elementary symmetric polynomials in squares.
By Deformation Principle in Proposition \ref{DefPrin}, we can conclude the following presentation of quantum cohomology.

\begin{Th}\label{prefortypeBC}
For $G$ of type $B_n$ or $C_n$, we have
\begin{equation}
QH^*_{\mathbb{T}}(T^*\calB) =
\frac{\calO(T^\vee_{\reg})
[\hbar,x_1,\ldots,x_n,t_1,\ldots,t_n]
}{\left<\mathcal{E}_{2k}(\chi)-(-1)^{k}e_{k}(t^2), k=1,\ldots,n\right>}.
\end{equation}
\end{Th}

\subsection{Type $D_n$}
The weight lattice is naturally recognized as
$$\bigg\{\lambda_1\mathbf{e}_1+\cdots+
\lambda_n\mathbf{e}_n\,\bigg|\,\lambda_i\in \frac{1}{2}\mathbb{Z}, \lambda_i-\lambda_j\in \mathbb{Z}. \bigg\}$$
with the set of simple coroots
$$\alpha^\vee_i=\mathbf{e}^\vee_i-\mathbf{e}^\vee_{i+1}\quad (i=1,\ldots,n-1)\qquad
\alpha^\vee_n=\mathbf{e}^\vee_{n-1}+\mathbf{e}^\vee_{n}.$$
We denote 
$$q_i=q^{\mathbf{e}^\vee_i},\qquad x_i=D_{-\mathbf{e}_i},\qquad\chi_i=\Delta_{-\mathbf{e}_i},\qquad t_i=-\mathbf{e}_i.$$
With respect to the ordered set 
 $$W^{-\mathbf{e}_1}=
 \big\{
 \id, s_1, {\ldots}, s_{n\texttt{-}1}s_{n\texttt{-}2}{\cdot}{\cdot}{\cdot} s_1,
 s_ns_{n\texttt{-}2}{\cdot}{\cdot}{\cdot} s_1,
 s_n{\cdot}{\cdot}{\cdot} s_1,
 s_{n\texttt{-}2}s_n{\cdot}{\cdot}{\cdot} s_1,
 {\ldots},
 s_1{\cdot}{\cdot}{\cdot} s_{n\texttt{-}2} s_n{\cdot}{\cdot}{\cdot} s_1
 \big\},$$
we have
\begin{equation}\label{eq:MatD}
M(\chi):=\Theta(-\mathbf{e}_1)=\left[
\begin{array}{@{}c@{\,}c@{\,}c|c@{\,}c@{\,}c@{}}
\chi_1 & \cdots&\frac{\hbar}{1-q_1/q_n}&
  \frac{\hbar}{1-q_1q_n} & \cdots & 0\\[1ex]
\vdots & \ddots&\vdots&
  \vdots&\reflectbox{$\ddots$}&\vdots\\[1ex]
\frac{\hbar}{1-q_n/q_1} & \cdots & \chi_n&
  0 & \cdots & \frac{\hbar}{1-q_nq_1}\\
    [1ex]\hline\rule{0pc}{3ex}
\frac{\hbar}{1-q_n^{-1}q_1^{-1}} & \cdots &0 &
  -\chi_n&\cdots & \frac{\hbar}{1-q_1/q_n}\\[1ex]
\vdots&\reflectbox{$\ddots$}&\vdots&
  \vdots & \ddots&\vdots\\[1ex]
0 & \cdots & \frac{\hbar}{1-q_1^{-1}q_n^{-1}}&
  \frac{\hbar}{1-q_n/q_1} & \cdots & -\chi_1
\end{array}\right].
\end{equation}
Here in the upper right (resp. lower left) block, the $(i,j)$-entry is given by $\frac{\hbar}{1-q_iq_{n+1-j}}$ (resp. $\frac{\hbar}{1-q_{n+1-i}^{-1}q_{j}^{-1}}$) when $i+j\neq n+1$ and $0$ otherwise. 
Similar to the case of type $A$, we have
$$\det\bigg(y 1_{2n}+M(\chi)\bigg)=\prod_{i=1}^n (y^2-t_i^2).$$
Expand the characteristic polynomial $$\det\bigg(y 1_{2n}+M(\chi)\bigg)=
y^{2n}+\mathcal{E}_1(\chi)y^{2n-1}+\cdots+\mathcal{E}_{2n-1}(\chi)y+\mathcal{E}_{2n}(\chi).$$
Thus for each $k\in \{1,\ldots, n\}$, $\mathcal{E}_{2k-1}(\chi)=0$ and the following relation holds in $QH_{\bbT}^*(T^*\calB)$,
$$\mathcal{E}_{2k}(\chi)=(-1)^ke_k(t^2).$$
For type $D_n$, the polynomial
$e_n(t)$ is also $W$-invariant.
In particular, 
$$(-1)^n\mathcal{E}_{2n}(\chi)=(-1)^n\det M(\chi)=e_n(t^2)=e_n(t)^2.$$
Consider
\begin{equation}
A(\chi)
=\left[\begin{array}{@{}c@{\,}c@{\,}c@{\,}c@{}}
\chi_1 &
\frac{\hbar(1-q_1)(1+q_2)}{(1-q_1/q_2)(1-q_1q_2)}&\cdots&
\frac{\hbar(1-q_1)(1+q_n)}{(1-q_1/q_n)(1-q_1q_n)}\\[1ex]
\frac{\hbar(1-q_2)(1+q_1)}{(1-q_2/q_1)(1-q_1q_2)}&\chi_2 & \cdots &
\frac{\hbar(1-q_2)(1+q_n)}{(1-q_2/q_n)(1-q_2q_n)}\\[1ex]
\vdots&\vdots&\ddots&\vdots\\
\frac{\hbar(1-q_n)(1+q_1)}{(1-q_n/q_1)(1-q_1q_n)}&
\frac{\hbar(1-q_n)(1+q_2)}{(1-q_n/q_2)(1-q_2q_n)}&\cdots&\chi_n
\end{array}\right].
\end{equation}

\begin{Lemma}\label{Lemma:M=detAdetA}
We have
$$
\det M(\chi)=\det A(\chi)\cdot
\det A(-\chi).$$  
\end{Lemma}
\begin{proof}
Denote
$$
L = \left[\begin{array}{@{}c@{\,}c@{\,}c@{\,\,}|@{\,}c@{\!\!}c@{\!\!}c@{}}
1  &&  &  &&-q_1^{-1}\\[-1ex]
&\ddots&&&\reflectbox{$\ddots$}\\
  &&1  &-q_n^{-1}& \\\hline
-q_1&&  &  &&-1\\[-1ex]
&\ddots&&&\reflectbox{$\ddots$}\\
  &&-q_n&-1 &&
\end{array}\right],\qquad
R=
\left[\begin{array}{c@{\,\,\,}c@{\,\,\,}c@{\,\,\,}|c@{}c@{}c@{}}
1  &&  &q_1^{-1} && \\[-1ex]
&\ddots&&&\ddots\\
  &&1  &  && q_n^{-1}\\\hline
  &&q_n &  && -1\\[-1ex]
&\reflectbox{$\ddots$}&&&\reflectbox{$\ddots$}\\[-1ex]
q_1 &&  &-1 &&
\end{array}\right].$$
By direct computation, we have
$$\frac{1}{2}LM(\chi)R
=\left[
\begin{matrix}
A(\chi)\\
&A(-\chi)^{\tt t}
\end{matrix}\right].$$
Note that
$$\frac{1}{2}RL
=\left[\begin{array}{@{}c@{\,\,\,}c@{\,\,\,}c@{\,\,\,}|@{\,}c@{\!\!\!}c@{\!\!\!}c@{}}
  &&  &  &&-q_1^{-1}\\[-1ex]
&&&&\reflectbox{$\ddots$}\\
  &&  &-q_n^{-1}& \\\hline
&& q_n  &  &&\\[-1ex]
&\reflectbox{$\ddots$}&&&\\
q_1 &&
\end{array}\right]$$
whose determinant is $1$.
As a result, we get the assertion.
\end{proof}

Consider the expansion of
$$\det A(\chi)=\sum_{\sigma\in \mathfrak{S}_n}
(-1)^{\ell(\sigma)}m_\sigma,\qquad m_\sigma=A(\chi)_{1\sigma(1)}\cdots A(\chi)_{n\sigma(n)}.$$

\begin{Lemma}
We have
$$\det A(\chi)=\sum_{\sigma}
(-1)^{\ell(\sigma)}m_\sigma$$
with $\sigma$ going through permutations whose nontrivial cycles all have even lengths.
\end{Lemma}

\begin{proof}
If $\sigma$ is a permutation having a non-trivial odd cycle, we define $\sigma'$ by inverting the unique odd cycle of $\sigma$ which contains the smallest index.
Since the correspondence $\sigma\mapsto \sigma'$ is an involution, it suffices to show
$$m_{\sigma}+m_{\sigma'}=0.$$
We can assume without loss of generality that the odd cycle is $(1\cdots k)$ for odd $k$.
Note that
\begin{align*}
{A(\chi)_{12}A(\chi)_{23}\cdots A(\chi)_{k1}}
& = \prod_{i}
\frac{\hbar(1-q_i)(1+q_{i+1})}{(1-q_i/q_{i+1})(1-q_iq_{i+1})}\\
& = \prod_{i}
\frac{\hbar(1-q_i^2)q_i}{(q_{i+1}-q_i)(1-q_iq_{i+1})}\\
& = (-1)^k {A(\chi)_{1k}A(\chi)_{21}\cdots A(\chi)_{k,k-1}}.
\end{align*}
where the indices are understood as elements of $\mathbb{Z}/k\mathbb{Z}$.
So we can deduce
$m_{\sigma}+m_{\sigma'}=0$.
\end{proof}

In particular, we have
$$\det A(\chi)=\det \big({-}A(-\chi)\big)
=(-1)^n\det A(-\chi).$$
Combining Lemma \ref{Lemma:M=detAdetA}, we have
$$\big(\det A(\chi)\big)^2=(-1)^n\det\big(M(\chi)\big)=e_n(t)^2.$$
Note that $e_n(t)\in H_G^*(\pt)$ and we can view $\det A(\chi)$ an element of $QH^*_{\mathbb{G}}(T^*\calB)$, see Remark \ref{rmk:TequivtoG}. 
Since $QH^*_{\mathbb{G}}(T^*\calB)$ is a domain by Proposition \ref{thm:BMO} (\ref{thm:BMO(3)}), we have $\det A(\chi)=\pm e_n(t)$. 
As $\det A(\chi)=e_n(x)\bmod \hbar$, the only possibility is $\det A(\chi)=e_n(t)$.

\begin{Th}\label{prefortypeD}
For $G$ of type $D_n$, we have
\begin{equation}
QH^*_{\mathbb{T}}(T^*\calB) =
\frac{\calO(T^\vee_{\reg})
[\hbar,x_1,\ldots,x_n,t_1,\ldots,t_n]
}{\left<\begin{array}{l}
\mathcal{E}_{2k}(\chi)-(-1)^{k}e_{k}(t^2),k=1,\ldots,n-1\\
\det A(\chi)-e_n(t)
\end{array}\right>}.
\end{equation}
\end{Th}

\appendix

\section{Toda limit}\label{sect:Todalimit}

The equivariant quantum cohomology of $\calB$ can be obtained from that of $T^*\calB$, by taking a Toda limit as follows.
\begin{Def}
The \emph{Toda limit} of a class  
$\gamma\in QH^*_{\mathbb{G}}(T^*\calB)$
 is defined to be
$$\tdlim\gamma =
\lim_{\hbar\to\infty}\gamma(\hbar,\hbar^{-2\rho}q)\in QH_G^*(\calB),$$
where $\gamma(\hbar,\hbar^{-2\rho}q)$ denotes the class obtained from $\gamma$ by the substitution
\[q^\beta\mapsto \hbar^{-\left<2\rho,\beta\right>}q^\beta.\]
When $\gamma\in QH^*_{\mathbb{T}}(T^*\calB)$,  $\gamma \in QH^*_T(\calB)$ is also called the Toda limit. 
 \end{Def} 
 
\begin{Prop}[{\cite[Proposition 8.1]{MR2782198}}]
\label{Prop:TodaLimit}
For any classes $\gamma_1,\gamma_2 \in QH^*_{\mathbb{G}}(T^*\calB)$ whose Toda limits both converge, we have
   \[\big(\tdlim\gamma_1\big)*
   \big(\tdlim\gamma_2\big)
   =\tdlim(\gamma_1*\gamma_2)\]
These also hold if we replace $\mathbb{G}$ by $\bbT$.
\end{Prop}

\begin{Rmk}  Proposition 8.1 of \cite{MR2782198} only  stated for the case when $\gamma_1$ is a divisor.
Nevertheless, its proof therein works in general.
This also follows from the fact that $QH^*_{\mathbb{G}}(T^*\calB)$ is generated by divisors.
\end{Rmk}

By taking the Toda limit, we can obtain a ring presentation of the equivariant cohomology $QH^*_T(\calB)$ from $QH^*_{\bbT}(T^*\calB)$. 
Below we discuss such presentation for $\calB$ of classical type, and compare them with the nice description by using 
integrals of motions of the Toda lattice for the Langlands-dual Lie group by Kim \cite{MR1680543}.

\subsection*{Type $A_{n-1}$} Note $\rho=(n-1)\mathbf{e}_1+(n-2)\mathbf{e}_2+\cdots+\mathbf{e}_{n-1}.$
We conjugate \eqref{eq:MatA} by $$\operatorname{diag}\big(1,-\hbar,\dots,(-\hbar)^{n-1}\big).$$
Note that
$$
\tdlim
\frac{(-1)^{i-j}\hbar^{1+i-j}}{1-q_i/q_j}
=\lim_{\hbar\to\infty}\frac{(-1)^{j-i}\hbar^{1+i-j}}{1-\hbar^{2(i-j)}q_i/q_j}
=\begin{cases}
q_j/q_{j+1}, & i=j+1,\\
-1, & j=i+1,\\
0 & \text{otherwise}.
\end{cases}$$
After taking the Toda limit, we obtain the following matrix
\begin{equation}\label{eq:MatAcomp}
\left[
\begin{matrix}
x_1 &
-1&
0&
\cdots&
0\\[1ex]
q_1/q_2 &
x_2 &
-1 &
\cdots &
0\\[1ex]
0 &
q_2/q_3 &
x_3 &
\cdots &
0\\[1ex]
\vdots&\vdots&\vdots&\ddots&\vdots\\[1ex]
0&
0&
0&
\cdots &
x_n
\end{matrix}\right].
\end{equation}
By Proposition \ref{Prop:TodaLimit}, this covers the representation of $QH^*_T(\calB)$ of type $A_{n-1}$ in \cite{MR1680543}.

\subsection*{Type $B_{n}$}
Note $\rho = (n-\tfrac{1}{2})\mathbf{e}_1+(n-\tfrac{3}{2})\mathbf{e}_2+\cdots+\tfrac{1}{2}\mathbf{e}_n.$
We conjugate \eqref{eq:MatBC} by $$\operatorname{diag}\big(
\hbar^{-n+1/2},\dots,\hbar^{-1/2},
\hbar^{1/2},\dots,\hbar^{n-1/2}
\big).$$
By taking the Toda limit entry-wise, we obtain the following matrix
\begin{equation}\label{eq:MatBcomp}
\arraycolsep=1.4pt
\def\arraystretch{0}
\left[
\begin{array}{cccc|cccc}
x_1 & 1&&&&\\[1ex]
-q_1/q_2& x_2 &\ddots&&&&\\[1ex]
&\ddots &\ddots & 1\\[1ex]
&&-q_{n-1}/q_n& x_n& 2 \\
    [1ex]\hline\rule{0pc}{3ex}
&&&-2q_{n}^2& -x_n& 1 \\[1ex]
&&&&-q_{n}/q_{n-1}& \ddots&\ddots \\[1ex]
&&&&&\ddots&-x_2 & 1\\[1ex]
&&&&&&-q_1/q_2 & -x_1\\[1ex]
\end{array}\right].
\end{equation}

We need the following lemma.

\begin{Lemma}[\cite{MR0114826}]\label{detoftridiag}
The determinant of a tridiagonal matrix is given by 
$$\det\left[
\begin{matrix}
x_1 & a_{1} & 0& \cdots& 0\\[1ex]
b_{1} & x_2 & a_{2} & \cdots & 0\\[1ex]
0 & b_{2} & x_3 & \cdots & 0\\[1ex]
\vdots&\vdots&\vdots&\ddots&\vdots\\[1ex]
0&0&0& \cdots & x_n
\end{matrix}\right]
=\sum_{\sigma}\bigg(\prod_{i=\sigma(i)}x_i
\prod_{\sigma(i)=i+1}(-a_{i}b_i)\bigg),$$
where $\sigma$ goes over involutions over $\{1,\ldots,n\}$ such that $|\sigma(i)-i|\leq 1$ for any $1\leq i\leq n$.
\end{Lemma}

As a corollary, replacing $(a_i,b_i)$ by $(-a_i,-b_i)$ for any $1\leq i\leq n$ does not affect the characteristic polynomial of a tridiagonal matrix.
Using this Lemma \ref{detoftridiag}, by suitable re-assigning sign for \eqref{eq:MatBcomp} and rearranging columns and rows if necessary, we recover the matrix in $\mathfrak{sp}_{n}$ obtained in \cite{MR1680543}.

\subsection*{Type $C_n$} Note $\rho = n\mathbf{e}_1+(n-1)\mathbf{e}_2+\cdots+\mathbf{e}_n.$
We can conjugate via the same diagonal matrix
$$\operatorname{diag}\big(
\hbar^{-n+1/2},\dots,\hbar^{-1/2},
\hbar^{1/2},\dots,\hbar^{n-1/2}
\big).$$
Since the $p$ is different from type $B$, we obtain the following matrix
\begin{equation}\label{eq:MatCcomp}
\arraycolsep=1.4pt
\def\arraystretch{0}
M=\left[
\begin{array}{@{}cccc|cccc@{}}
x_1 & 1&&&&\\[1ex]
-q_1/q_2& x_2 &\ddots&&&&\\[1ex]
&\ddots &\ddots & 1\\[1ex]
&&-q_{n-1}/q_n& x_n& 1 \\
    [1ex]\hline\rule{0pc}{3ex}
&&&-q_{n}& -x_n& 1 \\[1ex]
&&-q_{n-1}&&-q_{n}/q_{n-1}& \ddots&\ddots \\[1ex]
&\reflectbox{$\ddots$}&&&&\ddots&-x_2 & 1\\[1ex]
-q_1&&&&&&-q_1/q_2 & -x_1\\[1ex]
\end{array}\right].
\end{equation}

\begin{Lemma}
  
The polynomial $y\det(y1_{2n}+M)$ is the characteristic polynomial of the following matrix
\begin{equation}\label{eq:MatCcomp2}
\arraycolsep=1.4pt
\def\arraystretch{0}
\left[
\begin{array}{@{}c@{\,\,}c@{\,\,}c@{\,\,}c|c|c@{\,}c@{\,}c@{\,}c@{}}
x_1 & -1&&&&\\[1ex]
q_1/q_2& x_2 &\ddots&&&\\[1ex]
&\ddots &\ddots & -1&\\[1ex]
&&q_{n-1}/q_n& x_n&-1/2\\
    [1ex]\hline\rule{0pc}{3ex}
&&&q_{n}& 0& 1/2 \\
    [1ex]\hline\rule{0pc}{3ex}
&&&&-q_n& -x_n& 1 \\[1ex]
&&&&&-q_{n}/q_{n-1}& \ddots&\ddots \\[1ex]
&&&&&&\ddots&-x_2 & 1\\[1ex]
&&&&&&&-q_1/q_2 & -x_1\\[1ex]
\end{array}\right].
\end{equation}
\end{Lemma}

\begin{proof} We just give a sketch here.
By expanding the determinant, the characteristic polynomial $\det(y1_{2n}+M)$ equals 
\[
q_1+q_2
\det\big[\begin{subarray}{c}
y+x_1\\[0.5ex]
  \phantom{y-x_0}
\end{subarray}
\begin{subarray}{c}
  \phantom{y-x_0}\\[0.5ex]
y-x_1
\end{subarray}\big]
+q_3
\det\left[
\begin{subarray}{c}
y+x_1\\[0.5ex]
-q_1/q_2\\[0.5ex]
  \phantom{y-x_0}\\[0.5ex]
  \phantom{y-x_0}\\[0.5ex]
\end{subarray}\,\,
\begin{subarray}{c}
1\\[0.5ex]
y+x_2\\[0.5ex]
  \phantom{y-x_0}\\[0.5ex]
  \phantom{y-x_0}\\[0.5ex]
\end{subarray}\,\,
\begin{subarray}{c}
  \phantom{y-x_0}\\[0.5ex]
  \phantom{y-x_0}\\[0.5ex]
y-x_2\\[0.5ex]
-q_1/q_2\\[0.5ex]
\end{subarray}\,\,
\begin{subarray}{c}
  \phantom{y-x_0}\\[0.5ex]
  \phantom{y-x_0}\\[0.5ex]
1\\[0.5ex]
y-x_1\\[0.5ex]
\end{subarray}
\right]+\cdots
\]
Let us write $yq_i$ as
$$\frac{1}{2}(q_i/q_{i+1})\cdots (q_{n-1}/q_n)q_n(y+x_i)
+\frac{1}{2}(q_i/q_{i+1})\cdots (q_{n-1}/q_n)q_n(y-x_i).
$$
By Lemma \ref{detoftridiag} and tedious computation, we can conclude that
$y\det(y1_{2n}+M)$ coincides with the characteristic polynomial of the matrix \eqref{eq:MatCcomp2}.
\end{proof}

By suitable re-assigning sign for \eqref{eq:MatCcomp2} and rearranging columns and rows if necessary, we recover the matrix in $\mathfrak{so}_{2n+1}$ obtained in \cite{MR1680543}.

\subsection*{Type $D_n$}
 Note $\rho = (n-1)\mathbf{e}_1+(n-2)\mathbf{e}_2+\cdots+\mathbf{e}_{n-1}.$
We can conjugate the diagonal matrix
$$\operatorname{diag}\big(
\hbar^{-n+1},\dots,-\hbar^{-1},1,1,
\hbar,\dots,\hbar^{n-1}
\big).$$
We obtain the following matrix
\begin{equation}\label{eq:MatDcomp}
\arraycolsep=1.4pt
\def\arraystretch{0}
\left[
\begin{array}{@{}cc@{\,}c@{\,}c|cccc@{}}
x_1 & 1&&&&\\[1ex]
-q_1/q_2& x_2 &\ddots&&&&\\[1ex]
&\ddots &\ddots & 1&1\\[1ex]
&&-q_{n-1}/q_n& x_n& &1 \\
    [1ex]\hline\rule{0pc}{3ex}
&&-q_{n-1}q_n&& -x_n& 1 \\[1ex]
&&&-q_{n-1}q_n&-q_{n}/q_{n-1}& \ddots&\ddots \\[1ex]
&&&&&\ddots&-x_2 & 1\\[1ex]
&&&&&&-q_1/q_2 & -x_1\\[1ex]
\end{array}\right].
\end{equation}
By expanding determinant directly, it is not hard to show that it has the same characteristic polynomial as that of
\begin{equation}\label{eq:MatDcomp2}
\arraycolsep=1.4pt
\def\arraystretch{0}
\left[
\begin{array}{@{}cc@{\,}c@{\,}c|cccc@{}}
x_1 & -1&&&&\\[1ex]
q_1/q_2& x_2 &\ddots&&&&\\[1ex]
&\ddots &\ddots & -1&-1\\[1ex]
&&q_{n-1}/q_n& x_n& &1 \\
    [1ex]\hline\rule{0pc}{3ex}
&&q_{n-1}q_n&& -x_n& 1 \\[1ex]
&&&-q_{n-1}q_n&-q_{n}/q_{n-1}& \ddots&\ddots \\[1ex]
&&&&&\ddots&-x_2 & 1\\[1ex]
&&&&&&-q_1/q_2 & -x_1\\[1ex]
\end{array}\right].
\end{equation}
which is the matrix for $\mathfrak{so}_{2n}$ in \cite{MR1680543}.

 \section{Proof of Theorem \ref{combforEk}}\label{sec:Acomb}
 For type $A_{n-1}$,  the Weyl group $W$ is given by the symmetric group $\mathfrak{S}_{n}$. 
\begin{Lemma}\label{lemmaof2}
For $n\geq 2$, we have
\begin{equation}\label{eq:lemmaof2}
\sum_{\begin{subarray}{c}
\sigma\in\mathfrak{S}_{n}\\
\sigma(1)=1
\end{subarray}} 
\frac{1}{t_{\sigma(1)}{-}t_{\sigma(2)}}\cdots
\frac{1}{t_{\sigma(n{-}1)}{-}t_{\sigma(n)}}
\frac{1}{t_{\sigma(n)}{-}t_{\sigma(1)}}
=\begin{cases}
\frac{-1}{(t_1-t_2)^2},\!\!\!&n=2,\\
0, & n>2.
\end{cases}
\end{equation}
\end{Lemma}
\begin{proof}
This assertion is obvious for $n=2$. Now we assume $n>2$. It suffices to show
\begin{equation}\label{eq:lemmaof}
\sum_{\sigma\in\mathfrak{S}_n} \frac{1}{t_{\sigma(1)}-t_{\sigma(2)}}\cdots
\frac{1}{t_{\sigma(n-1)}-t_{\sigma(n)}}
\frac{1}{t_{\sigma(n)}-t_{\sigma(1)}}
=0
\end{equation}
for $n>2$.
In fact, each term of the sum \eqref{eq:lemmaof} is invariant under the action of the cycle $(12\cdots n)$, so it implies \eqref{eq:lemmaof2}.
Following the idea of \cite{LaszloJanosAdd}, we consider the following class in  $H_{\bbT}^*(\calB)$,
$$\gamma= -\prod_{
\begin{subarray}{c}
i<j\\
j-i\neq 1,n-1
\end{subarray}}(x_i-x_j),$$
where we recall $x_i=D_{-\mathbf{e}_i}$ and $t_i=-\mathbf{e}_i$. 
By Atyiah--Bott localization theorem \cite{MR721448}, we have
\begin{equation*}
\int_{\calB}\gamma =\sum_{\sigma\in\mathfrak{S}_n}\frac{\gamma|_{\sigma}}{\prod_{i<j}(t_{\sigma(i)}-t_{\sigma(j)})}
\in H_{\bbT}^*(\pt)\end{equation*}
which is exactly the left-hand side of \eqref{eq:lemmaof}.
But $\deg_{\bbC} \gamma = \frac{n(n-1)}{2}-n<\dim \calB$, thus $\int_{\calB}\gamma=0$.
\end{proof}

\begin{Lemma}\label{antiCauchy}We have
\begin{equation}\label{eq:antiCauchy}
\det\left[
\begin{matrix}
0 &
\frac{1}{t_1-t_2}&
\frac{1}{t_1-t_3}&
\cdots &
\frac{1}{t_1-t_n}\\[1ex]
\frac{1}{t_2-t_1} &
0 &
\frac{1}{t_2-t_3} &
\cdots &
\frac{1}{t_2-t_n}\\[1ex]
\frac{1}{t_3-t_1} &
\frac{1}{t_3-t_2} &
0 &
\cdots &
\frac{1}{t_3-t_n}\\[1ex]
\vdots&\vdots&\vdots&\ddots&\vdots\\[1ex]
\frac{1}{t_n-t_1} &
\frac{1}{t_n-t_2} &
\frac{1}{t_n-t_3} &
\cdots &
0
\end{matrix}\right]=\sum_{\pi}\prod_{\pi(i)=j>i}\frac{1}{(t_{i}-t_{j})^2}
\end{equation}
where the sum is over all perfect matchings $\pi$ of $\{1,\ldots,n\}$. 
\end{Lemma}
\begin{proof}Let us consider the expansion of the above determinant
$$\sum_{\sigma\in \mathfrak{S}_n} (-1)^{\ell(\sigma)}m_{\sigma}.$$
Firstly since the diagonal entries vanish, we have $m_{\sigma}=0$ if $\sigma$ has a fixed point.
Notice that the right-hand side of \eqref{eq:antiCauchy} equals the sum over permutations $\sigma$ such that $\sigma^2=\id$.

For any permutation $\sigma$ such that $\sigma^2\neq \id$, the cycle decomposition of $\sigma$ must have a cycle of length $>2$. 
Let $\eta_{\sigma}$ be the cycle of $\sigma$ of length $>2$ containing the smallest index, and denote by $A_{\sigma}$ the set of indices of $\eta_\sigma$.
Let $[\sigma]$ be the set of permutations $\sigma'$ such that $A_{\sigma}=A_{\sigma'}$ and
$\sigma(i)=\sigma'(i)$ for all $i\notin A_{\sigma}$.
In other words, $\sigma'\in [\sigma]$ is obtained by a permutation of indices of $\eta_{\sigma}$ of $\sigma$.
By Lemma \ref{lemmaof2} above,
$$\sum_{\sigma'\in [\sigma]} m_{\sigma'} = 0$$
Note that all permutations $\sigma'\in [\sigma]$ have the same sign. The proof is complete.
\end{proof}

\begin{Lemma}\label{Lemmawhenk=n}
Theorem \ref{combforEk} is true for $k=n$, i.e.
\begin{align*}
\det\left[\begin{matrix}
\chi_1 &
\frac{\hbar}{1-q_{1}/q_{2}}&
\cdots &
\frac{\hbar}{1-q_{1}/q_{n}}\\[1ex]
\frac{\hbar}{1-q_{2}/q_{1}} &
\chi_2 &
\cdots &
\frac{\hbar}{1-q_{2}/q_{n}}\\[1ex]
\vdots&\vdots
&\ddots&\vdots\\[1ex]
\frac{\hbar}{1-q_{n}/q_{1}} &
\frac{\hbar}{1-q_{n}/q_{2}} &
\cdots & \chi_n
\end{matrix}\right] &=\sum_{\pi}J(\pi)V(\pi),
\end{align*}
with the sum over all matchings $\pi$ over $\{1,\ldots,n\}$.
\end{Lemma}
\begin{proof}
By the expansion in polynomials in $\chi_i$'s,  the determinant is equal to
\begin{align*}
\sum_{A} {\chi}_{a_1}\cdots {\chi}_{a_{s}}
\det\left[\begin{matrix}
0 &
\frac{\hbar}{1-q_{b_1}/q_{b_2}}&
\cdots &
\frac{\hbar}{1-q_{b_1}/q_{b_t}}\\[1ex]
\frac{\hbar}{1-q_{b_2}/q_{b_1}} &
0 &
\cdots &
\frac{\hbar}{1-q_{b_2}/q_{b_t}}\\[1ex]
\vdots&\vdots
&\ddots&\vdots\\[1ex]
\frac{\hbar}{1-q_{b_d}/q_{b_1}} &
\frac{\hbar}{1-q_{b_d}/q_{b_2}} &
\cdots & 0
\end{matrix}\right]
\end{align*}
where the sum goes over all subsets $A=\{a_1,\ldots,a_s\}$ of $\{1,\ldots,n\}$ with complement $\{b_1<\cdots<b_t\}$.
Note that by Lemma \ref{antiCauchy}
\begin{align*}
&\quad
\det\left[\begin{matrix}
  0 &
  \frac{\hbar}{1-q_{b_1}/q_{b_2}}&
  \cdots &
  \frac{\hbar}{1-q_{b_1}/q_{b_t}}\\[1ex]
  \frac{\hbar}{1-q_{b_2}/q_{b_1}} &
  0 &
  \cdots &
  \frac{\hbar}{1-q_{b_2}/q_{b_t}}\\[1ex]
  \vdots&\vdots
  &\ddots&\vdots\\[1ex]
  \frac{\hbar}{1-q_{b_t}/q_{b_1}} &
  \frac{\hbar}{1-q_{b_t}/q_{b_2}} &
  \cdots & 0
  \end{matrix}\right]
\\& =
\hbar^t q_{b_1}\cdots q_{b_t}
\det\left[\begin{matrix}
  0 &
  \frac{1}{q_{b_2}-q_{b_1}}&
  \cdots &
  \frac{1}{q_{b_t}-q_{b_1}}\\[1ex]
  \frac{1}{q_{b_1}-q_{b_2}} &
  0 &
  \cdots &
  \frac{1}{q_{b_t}-q_{b_2}}\\[1ex]
  \vdots&\vdots
  &\ddots&\vdots\\[1ex]
  \frac{1}{q_{b_1}-q_{b_t}} &
  \frac{1}{q_{b_2}-q_{b_t}} &
  \cdots & 0
  \end{matrix}\right]
=\sum_{\pi} V(\pi),
\end{align*}
where the sum is over perfect matchings $\pi$ over $\{b_1<\ldots<b_t\}$. We can extend $\pi$ to a (not necessarily perfect) matching over $\{1,\ldots,n\}$ by setting $\pi(a_i)=a_i$ for $i=1,\ldots,s$.
Thus the determinant equals 
$$\sum J(\pi)V(\pi)$$
with sum over all matchings of $\pi$ over $\{1,\ldots,n\}$.
\end{proof}

\begin{proof}[Proof of Theorem \ref{combforEk}]
Similarly to the proof of Lemma \ref{Lemmawhenk=n}, the coefficient of $y^{n-k}$ in $\det\big(y 1_n+M(\chi)\big)$ is
\begin{align*}
\sum_{K}
\det\left[\begin{matrix}
\chi_{c_1} &
\frac{\hbar}{1-q_{c_1}/q_{c_2}}&
\cdots &
\frac{\hbar}{1-q_{c_1}/q_{c_k}}\\[1ex]
\frac{\hbar}{1-q_{c_2}/q_{c_1}} &
\chi_{c_2} &
\cdots &
\frac{\hbar}{1-q_{c_2}/q_{c_k}}\\[1ex]
\vdots&\vdots
&\ddots&\vdots\\[1ex]
\frac{\hbar}{1-q_{c_k}/q_{c_1}} &
\frac{\hbar}{1-q_{c_k}/q_{c_2}} &
\cdots & \chi_{c_k}
\end{matrix}\right]
\end{align*}
where the sum goes over all subsets $K=\{c_1<\cdots<c_k\}\subset\{1,\ldots,n\}$ of order $k$.
By Lemma \ref{Lemmawhenk=n}, it is equal to
$$\sum_K\sum_{\pi}J(\pi)V(\pi)$$
with $K$ a $k$-subset of $\{1,\ldots,n\}$ and $\pi$ going through the matchings of $K$.
This finishes the proof.
\end{proof}

\bibliographystyle{alpha}
\bibliography{mybibliography.bib}

\begin{thebibliography}{BLPW16}

\bibitem[AB84]{MR721448}
M.~F. Atiyah and R.~Bott.
\newblock The moment map and equivariant cohomology.
\newblock {\em Topology}, 23(1):1--28, 1984.

\bibitem[AF]{AFbook}
David Anderson and William Fulton.
\newblock {\em Equivariant Cohomology in Algebraic Geometry, Cambridge}.
\newblock Studies in Advanced Mathematics, Cambridge Univ. Press, Cambridge,
  2023.

\bibitem[AMSS17]{aluffi2017shadows}
Paolo Aluffi, Leonardo~C Mihalcea, J{\"o}rg Sch{\"u}rmann, and Changjian Su.
\newblock Shadows of characteristic cycles, {V}erma modules, and positivity of
  {C}hern-{S}chwartz-{M}acpherson classes of {S}chubert cells.
\newblock {\em to appear in Duke Math J., arXiv preprint arXiv:1709.08697},
  2017.

\bibitem[AS95]{MR1337131}
Alexander Astashkevich and Vladimir Sadov.
\newblock Quantum cohomology of partial flag manifolds {$F_{n_1\cdots n_k}$}.
\newblock {\em Comm. Math. Phys.}, 170(3):503--528, 1995.

\bibitem[BH18]{MR3817553}
Indranil Biswas and Jacques Hurtubise.
\newblock Geometry of {C}alogero-{M}oser systems.
\newblock {\em J. Math. Phys.}, 59(9):091403, 10, 2018.

\bibitem[BLPW16]{MR3594665}
Tom Braden, Anthony Licata, Nicholas Proudfoot, and Ben Webster.
\newblock Quantizations of conical symplectic resolutions {II}: category
  {$\mathcal{O}$} and symplectic duality.
\newblock {\em Ast\'{e}risque}, (384):75--179, 2016.
\newblock with an appendix by I. Losev.

\bibitem[BMO11]{MR2782198}
Alexander Braverman, Davesh Maulik, and Andrei Okounkov.
\newblock Quantum cohomology of the {S}pringer resolution.
\newblock {\em Adv. Math.}, 227(1):421--458, 2011.

\bibitem[Bor53]{MR51508}
Armand Borel.
\newblock Sur la cohomologie des espaces fibr\'{e}s principaux et des espaces
  homog\`enes de groupes de {L}ie compacts.
\newblock {\em Ann. of Math. (2)}, 57:115--207, 1953.

\bibitem[CF99]{MR1695799}
Ionu\c{t} Ciocan-Fontanine.
\newblock On quantum cohomology rings of partial flag varieties.
\newblock {\em Duke Math. J.}, 98(3):485--524, 1999.

\bibitem[CG10]{MR2838836}
Neil Chriss and Victor Ginzburg.
\newblock {\em Representation theory and complex geometry}.
\newblock Modern Birkh\"{a}user Classics. Birkh\"{a}user Boston, Ltd., Boston,
  MA, 2010.
\newblock Reprint of the 1997 edition.

\bibitem[CK99]{MR1677117}
David~A. Cox and Sheldon Katz.
\newblock {\em Mirror symmetry and algebraic geometry}, volume~68 of {\em
  Mathematical Surveys and Monographs}.
\newblock American Mathematical Society, Providence, RI, 1999.

\bibitem[Dan22]{Dan22}
Ivan Danilenko.
\newblock Quantum differential equation for slices of the affine
  {G}rassmannian.
\newblock {\em arXiv preprint arXiv:2210.17061}, 2022.

\bibitem[EFMV11]{MR2769318}
Pavel Etingof, Giovanni Felder, Xiaoguang Ma, and Alexander Veselov.
\newblock On elliptic {C}alogero-{M}oser systems for complex crystallographic
  reflection groups.
\newblock {\em J. Algebra}, 329:107--129, 2011.

\bibitem[Eti07]{MR2296754}
Pavel Etingof.
\newblock {\em Calogero-{M}oser systems and representation theory}.
\newblock Zurich Lectures in Advanced Mathematics. European Mathematical
  Society (EMS), Z\"{u}rich, 2007.

\bibitem[FN16]{LaszloJanosAdd}
L\'aszl\'o~M. Feh\'er and J\'anos Nagy.
\newblock Additive combinatorics using equivariant cohomology.
\newblock 2016.

\bibitem[FW04]{MR2072765}
W.~Fulton and C.~Woodward.
\newblock On the quantum product of {S}chubert classes.
\newblock {\em J. Algebraic Geom.}, 13(4):641--661, 2004.

\bibitem[Gin98]{MR1649626}
Victor Ginzburg.
\newblock Geometric methods in the representation theory of {H}ecke algebras
  and quantum groups.
\newblock In {\em Representation theories and algebraic geometry ({M}ontreal,
  {PQ}, 1997)}, volume 514 of {\em NATO Adv. Sci. Inst. Ser. C: Math. Phys.
  Sci.}, pages 127--183. Kluwer Acad. Publ., Dordrecht, 1998.
\newblock Notes by Vladimir Baranovsky [V. Yu. Baranovski\u{\i}].

\bibitem[GK95]{MR1328256}
Alexander Givental and Bumsig Kim.
\newblock Quantum cohomology of flag manifolds and {T}oda lattices.
\newblock {\em Comm. Math. Phys.}, 168(3):609--641, 1995.

\bibitem[Hec97]{MR1627113}
G.~J. Heckman.
\newblock Dunkl operators.
\newblock Number 245, pages Exp. No. 828, 4, 223--246. 1997.
\newblock S\'{e}minaire Bourbaki, Vol. 1996/97.

\bibitem[Hum78]{MR499562}
James~E. Humphreys.
\newblock {\em Introduction to {L}ie algebras and representation theory},
  volume~9 of {\em Graduate Texts in Mathematics}.
\newblock Springer-Verlag, New York-Berlin, 1978.
\newblock Second printing, revised.

\bibitem[Kam22]{J22}
Joel Kamnitzer.
\newblock Symplectic resolutions, symplectic duality, and {C}oulomb branches.
\newblock {\em arXiv preprint arXiv::2202.03913}, 2022.

\bibitem[Kim99]{MR1680543}
Bumsig Kim.
\newblock Quantum cohomology of flag manifolds {$G/B$} and quantum {T}oda
  lattices.
\newblock {\em Ann. of Math. (2)}, 149(1):129--148, 1999.

\bibitem[KMP21]{MR4295090}
Joel Kamnitzer, Michael McBreen, and Nicholas Proudfoot.
\newblock The quantum {H}ikita conjecture.
\newblock {\em Adv. Math.}, 390:Paper No. 107947, 53, 2021.

\bibitem[Liu13]{MR3184181}
Chiu-Chu~Melissa Liu.
\newblock Localization in {G}romov-{W}itten theory and orbifold
  {G}romov-{W}itten theory.
\newblock In {\em Handbook of moduli. {V}ol. {II}}, volume~25 of {\em Adv.
  Lect. Math. (ALM)}, pages 353--425. Int. Press, Somerville, MA, 2013.

\bibitem[Lus88]{MR972345}
George Lusztig.
\newblock Cuspidal local systems and graded {H}ecke algebras. {I}.
\newblock {\em Inst. Hautes \'{E}tudes Sci. Publ. Math.}, (67):145--202, 1988.

\bibitem[MNS22]{MR4465997}
Leonardo~C. Mihalcea, Hiroshi Naruse, and Changjian Su.
\newblock Left {D}emazure-{L}usztig operators on equivariant (quantum)
  cohomology and {K}-theory.
\newblock {\em Int. Math. Res. Not. IMRN}, (16):12096--12147, 2022.

\bibitem[MO19]{MR3951025}
Davesh Maulik and Andrei Okounkov.
\newblock Quantum groups and quantum cohomology.
\newblock {\em Ast\'{e}risque}, (408):ix+209, 2019.

\bibitem[MP15]{MR3421784}
Michael McBreen and Nicholas Proudfoot.
\newblock Intersection cohomology and quantum cohomology of conical symplectic
  resolutions.
\newblock {\em Algebr. Geom.}, 2(5):623--641, 2015.

\bibitem[MS13]{MR3095147}
Michael~B. McBreen and Daniel~K. Shenfeld.
\newblock Quantum cohomology of hypertoric varieties.
\newblock {\em Lett. Math. Phys.}, 103(11):1273--1291, 2013.

\bibitem[Mui60]{MR0114826}
Thomas Muir.
\newblock {\em A treatise on the theory of determinants}.
\newblock Dover Publications, Inc., New York, 1960.
\newblock Revised and enlarged by William H. Metzler.

\bibitem[Oko18]{MR3966746}
Andrei Okounkov.
\newblock On the crossroads of enumerative geometry and geometric
  representation theory.
\newblock In {\em Proceedings of the {I}nternational {C}ongress of
  {M}athematicians---{R}io de {J}aneiro 2018. {V}ol. {I}. {P}lenary lectures},
  pages 839--867. World Sci. Publ., Hackensack, NJ, 2018.

\bibitem[OP10]{MR2587340}
A.~Okounkov and R.~Pandharipande.
\newblock Quantum cohomology of the {H}ilbert scheme of points in the plane.
\newblock {\em Invent. Math.}, 179(3):523--557, 2010.

\bibitem[Pol19]{Polychronakos2019}
Alexios~P. Polychronakos.
\newblock Feynman's proof of the commutativity of the calogero integrals of
  motion.
\newblock {\em Annals of Physics}, 403:145--151, apr 2019.

\bibitem[SS93]{shastry1993super}
B~Sriram Shastry and Bill Sutherland.
\newblock Super lax pairs and infinite symmetries in the 1/r 2 system.
\newblock {\em Physical review letters}, 70(26):4029, 1993.

\bibitem[ST97]{MR1621570}
Bernd Siebert and Gang Tian.
\newblock On quantum cohomology rings of {F}ano manifolds and a formula of
  {V}afa and {I}ntriligator.
\newblock {\em Asian J. Math.}, 1(4):679--695, 1997.

\bibitem[Su16]{MR3439689}
Changjian Su.
\newblock Equivariant quantum cohomology of cotangent bundle of {$G/P$}.
\newblock {\em Adv. Math.}, 289:362--383, 2016.

\bibitem[Su17]{MR3595900}
Changjian Su.
\newblock Restriction formula for stable basis of the {S}pringer resolution.
\newblock {\em Selecta Math. (N.S.)}, 23(1):497--518, 2017.

\end{thebibliography}

\end{document}